\documentclass[11pt,a4paper, envcountsame,mathserif]{amsart}
\usepackage[usenames,dvipsnames]{color}

\usepackage[utf8]{inputenc}

 \usepackage[colorlinks,citecolor=blue,urlcolor=blue, linkcolor=blue, backref]{hyperref}

\usepackage[capitalise]{cleveref}		% Use \cref without environment name

 \usepackage{amsmath, amssymb, xspace}
 \usepackage{graphicx}

 \usepackage[all]{xy}

\usepackage{tikz}
\usetikzlibrary{arrows,snakes,positioning,backgrounds,shadows}

\newtheorem{theorem}{Theorem}[section]

\newtheorem{thm}[theorem]{Theorem}

\newtheorem{fact}[theorem]{Fact}

\newtheorem{proposition}[theorem]{Proposition}

\newtheorem{remark}[theorem]{Remark}

\newtheorem{prop}[theorem]{Proposition}

\newtheorem{claim}[theorem]{Claim}

\newtheorem{conjecture}[theorem]{Conjecture}

\newtheorem{lemma}[theorem]{Lemma}		

\newtheorem{corollary}[theorem]{Corollary}

\newtheorem{cor}[theorem]{Corollary}

\newtheorem{question}[theorem]{Question}

\theoremstyle{definition}
\newtheorem{definition}[theorem]{Definition}

\newcommand{\DII}{\Delta^0_2}
\newcommand{\NN}{{\mathbb{N}}}
\newcommand{\RR}{{\mathbb{R}}}

\newcommand{\QQ}{{\mathbb{Q}}}
\newcommand{\ZZ}{{\mathbb{Z}}}
\newcommand{\sub}{\subseteq}
\newcommand{\sN}[1]{_{#1\in \NN}}
\newcommand{\uhr}[1]{\! \upharpoonright_{#1}}
\newcommand{\ML}{Martin-L{\"o}f}
\newcommand{\SI}[1]{\Sigma^0_{#1}}
\newcommand{\PI}[1]{\Pi^0_{#1}}
\newcommand{\PPI}{\PI{1}}

\newcommand{\bi}{\begin{itemize}}
\newcommand{\ei}{\end{itemize}}
\newcommand{\bc}{\begin{center}}
\newcommand{\ec}{\end{center}}

\newcommand{\Halt}{{\ES'}}
\newcommand{\ES}{\emptyset}

\newcommand{\ria}{\rightarrow}
\newcommand{\tp}[1]{2^{#1}}
\newcommand{\ex}{\exists}
\newcommand{\fa}{\forall}
\newcommand{\lep}{\le^+}

\newcommand{\la}{\langle}
\newcommand{\ra}{\rangle}
\newcommand{\Kuc}{Ku{\v c}era}
\newcommand{\seqcantor}{2^{ \NN}}

\newcommand{\cantor}{\seqcantor}
\newcommand{\strcantor}{2^{ < \omega}}

\newcommand{\fao}[1]{\forall #1 \, }

\newcommand{\Opcl}[1]{[#1]^\prec}
\newcommand{\leT}{\le_{\mathrm{T}}}

\newcommand{\MLR}{\mbox{\rm \textsf{MLR}}}

\newcommand{\n}{\noindent}

\newcommand{\vsps}{\vspace{3pt}}

\newcommand{\vsp}{\vspace{6pt}}
\newcommand{\leb}{\mathbf{\lambda}}

\newcommand{\sss}{\sigma}
\newcommand{\aaa}{\alpha}

\DeclareMathOperator{\SAT}{SAT}
\DeclareMathOperator{\NP}{NP}
\DeclareMathOperator{\MPT}{MPT}

\newcommand{\lland}{\, \land \, }

\newcommand \seq[1]{{\left\langle{#1}\right\rangle}}

\newcommand\+[1]{\mathcal{#1}}

\newcommand{\wt}{\widetilde}
\newcommand{\ol}{\overline}
\newcommand{\ul}{\underline}

\newcommand{\lra}{\leftrightarrow}
\newcommand{\LR}{\Leftrightarrow}
\newcommand{\RA}{\Rightarrow}
\newcommand{\LA}{\Leftarrow}

\newcommand{\rapf}{\n $\RA:$\ }
\newcommand{\lapf}{\n $\LA:$\ }

\newcommand{\sssl}{\ensuremath{|\sigma|}}

\newcommand{\dom}{\ensuremath{\mathrm{dom}}}

\def\uh{\upharpoonright}

\DeclareMathOperator{\Hom}{Hom}

\DeclareMathOperator{\id}{id}

  \newcommand{\SR}{\mbox{\rm \textsf{SR}}}
  
\newcommand{\frb}{\mathfrak{b}}
\newcommand{\frd}{\mathfrak{d}}
\DeclareMathOperator{\cov}{cover}
\DeclareMathOperator{\non}{non}
\DeclareMathOperator{\cof}{cofin}
\DeclareMathOperator{\add}{add}

\begin{document}

\title{Logic Blog 2015}

 \author{Editor: Andr\'e Nies}

\maketitle

%\begin{abstract}  %The 2015  logic blog has focussed on the following:
% \end{abstract}

 {
The Logic Blog is for
\bi \item rapidly announcing    results related to logic
\item putting up results and their proofs for further research
\item archiving results for later use
\item getting feedback before submission to   a journal.   \ei

Each year's  blog is    posted on arXiv.org shortly after the year has ended.
\vsp
\begin{tabbing}

  \href{http://arxiv.org/abs/1504.08163}{Logic Blog 2014} \ \ \ \   \= (Link: \texttt{http://arxiv.org/abs/1504.08163})  \\

   \href{http://arxiv.org/abs/1403.5719}{Logic Blog 2013} \ \ \ \   \= (Link: \texttt{http://arxiv.org/abs/1403.5719})  \\

    \href{http://arxiv.org/abs/1302.3686}{Logic Blog 2012}  \> (Link: \texttt{http://arxiv.org/abs/1302.3686})   \\

 \href{http://arxiv.org/abs/1403.5721}{Logic Blog 2011}   \> (Link: \texttt{http://arxiv.org/abs/1403.5721})   \\

 \href{http://dx.doi.org/2292/9821}{Logic Blog 2010}   \> (Link: \texttt{http://dx.doi.org/2292/9821})  
     \end{tabbing}

\vsp

\n {\bf How does the Logic Blog work?}

\vsp

\n {\bf Writing and editing.}  The source files are in a shared dropbox.
 Ask Andr\'e (\email{andre@cs.auckland.ac.nz})  in order    to gain access.

\vsp

\n {\bf Citing.}  Postings can be cited.  An example of a citation is:

\vsp

\n  H.\ Towsner, \emph{Computability of Ergodic Convergence}. In  Andr\'e Nies (editor),  Logic Blog, 2012, Part 1, Section 1, available at
\url{http://arxiv.org/abs/1302.3686}.}

\vsp

\n {\bf Announcements on the wordpress front end.}  The Logic Blog has a \href{http://logicblogfrontend.hoelzl.fr/
}{front-end}  managed by Rupert H\"olzl.   

\n (Link: \texttt{http://logicblogfrontend.hoelzl.fr/})

\vsps
When you post source code on the logic blog in the dropbox, you can post a comment on the front-end alerting the community, and possibly summarising the result in brief.  The front-end is also good for posting questions. It allows MathJax.
 
The logic blog,  once it is on  arXiv.org,  produces citations on Google Scholar.
%\n  A. Taveneaux, \emph{Randomness Zoo}, Logic Blog, Section 6, available at
%
%\texttt{http://dl.dropbox.com/u/370127/Blog/Blog2011.pdf}.
\newpage
\tableofcontents

 \part{Randomness via Kolmogorov complexity}

\section{Yu: A trivial observation on mutual information }

\begin{definition}[Levin] For oracles $x,y$, 
$$I(x:y)=\log \sum_{n,m}2^{-K(\langle m,n\rangle)-K^x(m)-K^y(n)+K(m)+K(n)}.$$
\end{definition}

\begin{definition}[Levin]
$$I'(x:y)=\log \sum_{n}2^{-K^x(n)-K^y(n)+K(n)}.$$
\end{definition}

Note that $I'(x,y) = \infty$ implies $I(x,y) = \infty$. In \cite{HW12}, Hirschfeldt and Weber asked whether there is a non-trivial $x$ so that for any $y$ with $I(x:y)=\infty$, $y$ must compute $x\oplus \emptyset'$. Actually by their own proof, the answer is negative.

\begin{proposition}
If $x$ is not $K$-trivial, then  $\{y\mid I'(x:y)=\infty\}$ is comeager.
\end{proposition}
\begin{proof}
 Essentially due to Hirschfeldt and Weber.  For any $n$, there is some $k_n$ so that $K^x(x\uh k_n)\leq K(x\uh k_n)-n$. If $y$ is sufficiently generic, then for any $n$, there is some $i_n>n$ so that $K^y(x\uh k_{i_n})\leq i_n$. Then $$I'(x:y)\geq \log \sum_{n}2^{-K^x(x\uh k_{i_n})-K^y(x\uh k_{i_n})+K(x\uh k_{i_n})}\geq \log \sum_{n}2^{i_n-i_n}=\infty.$$
\end{proof}

\section{Nies: An analog of  the coincidence  of $\le_{LR} $ and  $\le_{LK}$}

Let $Y,B$ be sets. Recall that $Y \le_{LR} B$ if $\MLR^B \sub \MLR^Y$. Kjos-Hanssen, Miller and Solomon~\cite{Kjos.Miller.ea:11} showed that this LR-reducibility coincides with LK-reducibility, where $Y \le_{LK} B$ if $\fa x \, K^B(x) \lep K^Y(x)$. 
We will weaken both relationships by replacing the objects c.e.\ in $Y$  by objects computable in $Y$.  This is applied in a recent manuscript by Greenberg, Miller and Nies on subclasses of the $K$-trivials.

A computable measure machine   (c.m.m.)  is a prefix free machine $M$ such that $\leb \Opcl  {\dom M}$ is a computable real \cite[3.5.14]{Nies:book}.
\begin{definition}  \mbox{} 

\n (i) We write that  $Y \le_{wLR} B$ if $\MLR^B \sub \SR^Y$. 

\n (ii) $Y \le_{wLK} B$ if $\fa x \, K^B(x) \lep K_{M}(x)$ for each computable measure machine $M$ relative to $Y$. \end{definition}
Note that these relations  are not transitive. Barmpalias, Miller and Nies~\cite{Barmpalias.Miller.ea:12} have shown that $Y \le_{wLR} B$ iff $Y$ is c.e.\ traceable by $B$: there is a computable bound $h$ such that each function  $f \leT Y$ has an $h$-bounded trace c.e.\ in $B$.   $\Halt \le_{wLR} B$ means that $B$ is ``weakly LR-hard". For c.e.\ sets, array recursive is the same as c.e.\ traceable, and it is known that such sets can be properly low$_2$. Hence, by jump inversion for ML-random sets, some  weakly LR-hard ML-random $\DII$ set $B$  is properly high$_2$, and in particular not LR-hard. Every random set above a smart $K$-trivial in the sense of \cite{Bienvenu.Greenberg.ea:nd} is not OW-random, hence LR-hard, and in particular high. So the diamond class of   weak LR-hardness is properly contained in the $K$-trivials. In fact by the result in \cite{Barmpalias.Miller.ea:12} it is contained in the diamond class of JT-hardness, which was previously known to be properly contained in the $K$-trivials (see \cite[8.5]{Nies:book}).

We adapt the proof of the Kjos-Miller-Solomon result given as Theorem 5.6.5 in Section 5.6 of \cite{Nies:book}. Item numbers below refer to \cite{Nies:book}. We only give proofs when they are not straightforward adaptations.
The following is our new version of  {\bf 5.6.5}.
\begin{thm}  \label{thm: LRLK} Let $Y,B$ be sets.  We have $Y \le_{wLK} B$ $\LR$ $Y \le_{wLR} B$. \end{thm}
\begin{proof} \rapf This follows from the characterisation of ML-randomness via $K$ (Levin-Schnorr), and the characterisation of Schnorr-randomness via $K_M$ for c.m.m.\ $M$, all relativized appropriately. 

\lapf  
That implication depended on a number of foregoing results, some of them  in Section 5.1.

\n New version of  {\bf 5.1.10}, ``partially" relativized to $B$.
\begin{prop} \label{thm:char Low R}  The following are equivalent for a set~$A$.
\bi
\item  $Y \le_{wLR} B$
\item  For each  computable measure  machine~$M$ relative to $Y$,  there is a  $\SI 1 (B)$ set~$S$ such that  \begin{equation} \label{AM:ENS1}  \leb S < 1 \
\land\ \fa \,  z  [K_{M}(z) \le |z|-1 \rightarrow [z] \sub S].  \end{equation} \ei
\end{prop}

\vsp

\n New version of  {\bf 5.6.3} (which extends 5.1.10).
 \begin{lemma} \label{cor:Bjorns great result2} $Y\le_{wLR} B$  $\LR$   each  $\SI{1}(Y)$ class~$G$ such that  $\leb G <1$ \emph{and $\leb G$ is computable} in $Y$  is contained in   a $\SI{1}(B) $ class $S$ such that  $\leb S <1$.
\end{lemma}

For a function $f \colon \NN \ria \NN$, let $\mu_f$ be  the measure on $\mathcal{P}(\NN)$ given by $\mu_f(\{n\}) = \tp{-f(n)}$.  A set $I \sub \NN$ is called \emph{$f$-small}   if   $\mu_f (I)$ is finite.

\n New version of  {\bf 5.6.4}.

\begin{lemma} \label{lem:simpson's great}    $Y \le_{wLR} B$  $\RA$  for each  computable function~$f$,
 each~$f$-small~$Y$-c.e.\ set~$I$ \emph{such that $\mu_f(I)$ is computable in $Y$} is contained in an~$f$-small~$B$-c.e.\ set~$R$.  \end{lemma}

\begin{proof} As before, we use the  fact that  for   a sequence of real numbers $(a_n)\sN{n}$ such that  $0 \le a_n <1$ for each $n$,   we have

\begin{equation} \label{eqn:a i sum prod} \sum_{n=0}^\infty  a_n< \infty \LR  \prod_{n=0}^\infty  (1-a_n)>0. \end{equation}
  To see this  one works with  $g(x) = - \ln (1-x)$ for $x \in [0,1)_{\mathbb{R}}$. Since $e^y \ge 1+y$ for each $y\in \mathbb{R}$, we have  $x \le g(x)$  for each~$x$. On the other hand  $g'(x) = 1/(1-x)$, so $g'(0) =1 = \lim_{x\ria 0} (g(x) -g(0))/x$. Hence  there is $\varepsilon >0$ such that $g(x) \le 2x$ for each $x \in [0,\varepsilon)$.

In the present setting, we also need that if the sum is a (finite) computable real,  then so is  the product. We may assume that each tail sum   is less than $\epsilon$. It now suffices to verify  that  $\prod_{n=k}^\infty  (1-a_k) \le g(\sum_{n=k}^\infty a_k)$ for each $n$. 
 
As before we use  (\ref{eqn:a i sum prod}) to  infer   Lemma~\ref{lem:simpson's great} from  the implication ``$\RA$'' of Lemma~\ref{cor:Bjorns great result2}.  It suffices to observe that the class $P$ defined in the original version now has $Y$-computable positive measure because the product is $Y$-computable. Now let $G = \cantor - P$ and apply $\RA$ of~\ref{cor:Bjorns great result2}.    \end{proof}
%We may assume $f(n) >0$ for each~$n$, for we can replace~$f$ by $\lambda n . f(n)+1$ without changing the notion of~$f$-smallness. For each~$n$, let $g(n) = \sum_{i<n} f(i)$, and let $E_n$ be the clopen set $\{Z \colon \,\ex i \in [g(n), g(n+1)) \, [Z(i) \neq 0]\}$. Note that $\leb E_n =1-\tp{-f(n)}$. Also   $\leb (\bigcap_{n \in X} E_n) = \prod_{n \in X} \leb E_n$ for each $X\sub \NN$.
%(This is easy to check for finite~$X$; in the general case, $ \leb (\bigcap_{n \in X} E_n) = \lim_m  \leb (\bigcap_{n \in X,n \le m} E_n) = \lim_m  \prod_{n \in X, n\le m} \leb E_n = \prod_{n \in X} \leb E_n$.  We use that  the events $E_n$ are   independent in the language of probability theory.)\index{independent events}
%
%Since~$I$ is~$Y$-c.e., the class $P=\bigcap_{n \in I} E_n$ is $\Pi_1^{0}(Y)$. For we may assume that $I \neq \ES$, so there is a function $h \leT Y$ with range~$I$;  then    $P =\{Z \colon \, \fao k Z \in E_{h(k)}\}$. Since~$I$ is~$f$-small,
%by  (\ref{eqn:a i sum prod}) we have $\leb P = \prod_{n \in I} (1-\tp{-f(n)}) >0$. By Lemma~\ref{cor:Bjorns great result2}  choose a $\PI{1}(B) $ class $Q \sub P$ such that $\leb Q >0$. Let  $R = \{n\colon \,  Q\sub E_n \}$.
% Then~$R$ is~$B$-c.e., $I \sub R$, and~$R$ is~$f$-small, by (5.21) and because
%
% \bc $\prod_{n \in R} (1-\tp{-f(n)}) = \prod_{n \in R} \leb(E_n) = \leb \bigcap_{n \in R} E_n \ge \leb Q >0$, \ec by
% the independence  of the events $E_n$.

We can now complete the proof of \lapf of Thm.\ref{thm: LRLK}. Let~$f$ be the computable function given by $f(\la r,y\ra) = r$ (the book has the typo $2^r$ there). Let  $M$ be  a c.m.m.\ relative to $Y$.  The  set $I= \{\la \sssl, y\ra\colon \,  M(\sss) = y\}$ is a bounded request set relative to~$A$ and hence~$f$-small, and  $\mu_f(I)$ is computable in $Y$. So by   Lemma~\ref{lem:simpson's great}, $I$  is contained in   an~$f$-small~$B$-c.e.\ set $\wt R $. Let $R \sub \wt R$ be a bounded request set relative to~$B$ such that $\wt R -R$ is finite. Then, applying to~$R$   the  Machine Existence Theorem \cite[2.2.17]{Nies:book}  relative to~$B$, we may conclude that  $\fao y K^B(y)\lep K_M(y)$.  
\end{proof}
\part{Randomness via algorithmic tests}

\section{Downey, Nandakumar and Nies: \\   Multiple recurrence and randomness via algorithmic tests}

%Downey, Nandakumar and Nies worked at the Research Centre Coromandel and then at   Victoria University Wellington. They studied effective forms of multiple recurrence theorems. 
We determine the level of randomness needed for a point so that the multiple recurrence theorem of Furstenberg holds for iterations starting at the point. 
\subsection{Background in ergodic  theory} Let $(X, \+ B, \mu)$  be a probability space. A measurable  operator $T\colon  X \to X$    is called \emph{measure preserving} if   $\mu T^{-1}(A) = \mu A$ for each $ A \in \+ B$.

\begin{theorem}[Furstenberg strong multiple recurrence theorem; see~\cite{Furstenberg:2014} Thm.\ 7.15]  \label{thm:FurMRT} \  \\
\n  Let $(X, \+ B, \mu)$ be a probability space. Let $T_1, \ldots, T_k$ be commuting measure preserving operators on $X$. 
Let    $A \in \+B$ with $\mu A > 0$. We have 
\[ 0< \liminf_N \frac 1 N \sum_{n=1}^N  \mu ( \bigcap_{1 \le i \le k}  T_i^{-n}(A)) \]
\end{theorem}
\n As a consequence,    there is a positive measure set of points so that  the same number  of iterations   of each of the operators  $T_i$,  starting from each of the points, ends in $A$. 
\begin{cor}[Furstenberg  multiple recurrence theorem]  \label{prop: FurMRT2} \ \\ With the hypotheses of Thm.\ \ref{thm:FurMRT},  there is 
$n>0$ such that $0< \mu \bigcap_i T_i^{-n}(A)$.   \end{cor}

%We can   get  the condition    for $\mu$-a.e.\ $x \in A$  in   the general  setting of the  Furstenberg theorem.

For this paper the following variant of Cor.\  \ref{prop: FurMRT2} will matter.
\begin{cor} \label{prop:FurMRT3} With the hypotheses of Thm.\ \ref{thm:FurMRT}, for $\mu$-a.e.\ $x \in A$, there is an $n>0$ such that $x \in \bigcap_i T_i^{-n}(A)$. \end{cor}

This statement   clearly yields Cor.\ \ref{prop: FurMRT2} because it implies that $  \bigcap_i T_i^{-n}(A)$   has positive measure for some $n$.  Conversely,  let us  show that  Cor.\ \ref{prop: FurMRT2} yields   Cor.\  \ref{prop:FurMRT3}. 
Let  $R_n = \bigcap_i T_i^{-n}(A)$.  We recursively define a sequence $\seq {n_p}_{p< N}$ of numbers and a descending sequence $\seq {A_p}_{p< N}$ of sets, where $0<N \le \omega$.

Let $n_0 = 0$, and  $A_0 = A$. Suppose  $n_p$ and $A_p$ have been defined. If  $\mu A_p =0$ let $N=p+1$ and finish. Otherwise, let $n_{p+1}$ be the least $n > n_p$ such that $\mu  (A_p \cap R_n) >0$, and   let  $A_{p+1} = A_p - R_n$. 

Let $A_N= \bigcap_{p<N} A_p$. Then $\mu A_N =0$. This is clear if $N$ is finite. If $N= \omega$ and $\mu A_N >0$, by Cor.\ \ref{prop: FurMRT2}  there  is $n$ such that 
$\mu( R_n \cap A_N) > 0$. This contradicts the definition of $A_{p+1}$ where $n_p < n \le n_{p+1}$.

Since $\mu A_N =0$,  Cor.\  \ref{prop:FurMRT3}  follows.

We will be mainly interested in the   special case where  $T_i = V^i$ for a measure-preserving  operator $V$. 
\begin{cor} \label{cor:MRT} Let $(X, \+ B, \mu)$ be a probability space. Let $V$ be  a  measure preserving operator.  Let $A \in \+ B$  and $\mu A > 0$. For each $k$, for $\mu$-a.e.\ $x \in A$   there is $n$ such that $\fa i. {1 \le i \le k}  \, [x \in V^{-ni}( A)]$. 
\end{cor}

%A full      ``almost everywhere" version of Corollary \ref{cor:MRT} would assert that for $\mu$-a.e.\ $x$ there is $n$ such that  $\fa i. {1 \le i \le k}  \, [  S^{ni}x \in A]$. 
%Note that we can expect such a version   only    for ergodic operators, even if $k=1$. For, if $A$ is $S$-invariant,   then an iteration starting from $x \not \in A$ will never get into~$A$.

In fact we will mostly assume that   $(X, \+ B, \mu)$   is  Cantor space $\cantor$ with the product measure $\leb$.  In the following $X,Y,Z$ will denote elements of Cantor space. We will work with   the shift  $T$ as the measure preserving operator. Thus, $T(Z)$ is obtained by deleting the first entry of the bit  sequence $Z$. We note that this operator    is (strongly)  mixing, and hence strongly ergodic, namely, all of its powers are ergodic.
We will   write   $Z_n$ for $T^n (Z)$,   the tail of $Z$ starting at bit position $n$.  Thus, for any $\+ C \sub \cantor$,  $Z \in \cantor$ and $k \in \NN$, $Z \in T^{-k} (\+ C) \lra Z_k \in \+ C$.

For a set of strings $S \sub \strcantor$, by $\Opcl S$ we denote the open set $\{ Y  \in \cantor \colon \, \ex \sss \in S \, [\sss \prec Y]\}$. 
We write $\leb \Opcl S$ for the measure of this set, namely $\leb (\Opcl S)$

\subsection{The connection with algorithmic randomness}

For the remainder of the paper, we consider multiple recurrence for  closed sets.   Note that for (multiple) recurrence in the sense of Cor.\ \ref{prop: FurMRT2},   this  is not an essential restriction, because any set of positive measure contains a closed subset of positive measure. 
The following is our central definition. Let $T \colon \cantor \to \cantor$ denote the shift operator.
\begin{definition} Let $\+ P \sub \cantor$ be closed, and let  $Z \in \cantor$. We say that   $Z$ is  \emph{$k$-recurrent  in $\+ P$}  if  there is $n \ge 1$ such that 
\[ \tag{$\diamond$} Z \in \bigcap_{1 \le i \le k} T^{-ni} (\+ P). \]
We  say that $Z$ is \emph{multiply recurrent} in $\+ P$ if $Z$ is $k$-recurrent in $\+ P$  for each $k \ge 1$.
\end{definition}
In other words, $Z$ is $k$-recurrent in $\+ P$ if there is $n$ such that taking $n, 2n, \ldots, kn$   bits off $Z$ takes us into $\+ P$.

 We analyse how weaker and weaker  effectiveness conditions on $\+ P$ ensure  multiple recurrence when starting from a  sequence $Z$ that satisfies a stronger and stronger  randomness property for an algorithmic test notion. We begin with the strongest effectiveness condition, being clopen; in this case it is easily seen that   weak (or Kurtz)   randomness of $Z$  suffices. As  most general effectiveness condition  we will  consider being effectively closed (i.e.\ $\PI 1$); \ML-randomness turns out to be the appropriate notion. The proof will import some method from the  case of a clopen $\+ P$.  Note that $\PI 1$ subsets of Cantor space are often called \emph{$\PI 1$ classes}. For background on randomness notions see \cite[Ch.\ 3]{Nies:book} or \cite[Ch ?]{Downey.Hirschfeldt:book}.

\subsection{Multiple recurrence for weakly  random sequences} 
Recall that $Z$ is weakly  (or Kurtz) random if $Z$ is in no null $\PI 1$ class. 
\begin{prop} \label{prop: Kurtz} Let $\+ P \sub \cantor$  be a non-empty  clopen set. Each weakly   random bit sequence $Z$ is multiply recurrent  in $\+ P$.\end{prop}

\begin{proof}

Suppose    $Z$ is not $k$-recurrent in $\+ P$ for some $k \ge 1$. We define a null $\PPI$ class  $\+ Q$ containing $Z$. Let $n_0$ be least such that $\+ P = \Opcl F$ for some set of strings of length $n_0$.  Let $n_t= n_0(k+1)^t$ for $t \ge 1$. Let
\[ \+ Q = \bigcap_{t \in \NN}  \{ Y \colon   \bigvee_{1 \le i \le k} Y_{i n_t} \not \in \+ P\}. \]
 By definition  of $n_0$  the conditions in  the same disjunction are independent, so  we have \[ \leb (\bigvee_{1 \le i \le k} Y_{i n_t} \not \in \+ P ) =  1 - (\leb \+ P) ^k < 1.\] By definition of the $n_t$ for $t >0$, the class $\+ Q$ is the independent intersection of  such classes  indexed by~$t$. Therefore $\+ Q$ is null. Clearly $\+ Q$ is $\PPI$.
 
   By hypothesis $Z \in \+ Q$. So $Z$ is not weakly random.
\end{proof} 

\subsection{Multiple recurrence for     Schnorr random sequences} 
 
\begin{thm} Let $\+ P \sub \cantor$  be a $\PPI$ class such that   $0< p = \leb  \+ P$ and $p$ is a computable real. Each Schnorr   random   $Z$ is multiply recurrent  in $\+ P$.
\end{thm}
We note that   this  also follows from a particular effective version of  Furstenberg multiple recurrence (Cor.\ \ref{prop:FurMRT3}), as  explained in Remark~\ref{Rem:Rute} below. However, we prefer to give a direct proof avoiding Cor.\ \ref{prop:FurMRT3}.
\begin{proof} We extend the previous proof, working with an effective approximation $\+ B= \cantor - \+ P = \bigcup_s  \+ B_s$ where the $\+ B_s$ are clopen. We may assume that $\cantor -  \+ B_s = \Opcl {B_s}$ for some effectively given set $B_s$ of strings of length $s$. 

We fix an arbitrary $k \ge 1$ and show that $Z$ is    $k$-recurrent in $\+ P$. 
 Given $v \in \NN$ we will define a null $\PPI$ class $\+ Q_v \sub \cantor$ which plays a role similar to the class $\+ Q$  before. We also define an  ``error class'' $\+ G_v \sub \cantor$ that is $\SI 1$ uniformly in $v$. Further, $\leb \+ G_v$ is computable uniformly in $v$ and $\leb \+ G_v \le \tp {-v}$, so that $\seq { \+ G_v }\sN v$ is a Schnorr test.  If $Z$ passes this Schnorr test then $Z$ behaves essentially like a weakly random in the  proof of  Proposition~\ref{prop: Kurtz}, which shows that $Z$ is $k$-recurrent for  $\+ P$. 
 
For the details, given  $v \in \NN$, we define a computable sequence $\seq {n_t}$.  Let $n_0= 1$.
   Let $n= n_t \ge (k+1)n_{t-1}$ be so large that \bc $\leb (\+ B - \+ B_n) \le \tp{-t- v-k}$. \ec
   
As in the proof of Proposition~\ref{prop: Kurtz}, the  class
\[ \+ Q_v = \{ Y \colon \fa t  \bigvee_{1 \le i \le k} Y_{i n_t} \in \+ B_{n_t}\} \]
is  $\PPI$ and null. The ``error class" for  $v$ at stage $t$ is
\[ \+ G_{v}^t = \{ Y \colon \bigvee_{1 \le i \le k} Y_{i n_t} \in  \+ B - \+ B_{n_t}\}. \]
Notice  that $\leb \+ G^t_v \le k \tp{-t- v-k}$, and this measure is computable uniformly in $v,t$. Let $\+ G_v = \bigcup_t G^t_v$. Then $\leb \+ G_v$ is also uniformly computable in $v$, and bounded above by $\tp {-v}$, as required. 

If  $Z$ is Schnorr random, there is $v$ such that $Z \not \in \+ G_v$. Also, $Z \not \in \+ Q_v$, so that for some $t$ we have $Z_{i n_t} \in \+ P$ for each $i$ with  $1 \le i \le k$, as required.
\end{proof}

\subsection{Multiple recurrence for ML-random sequences}
For  general  $\PPI$ classes, the right level of randomness to obtain multiple recurrence is ML-randomness. We first remind the reader that  even the case of $1 $-recurrence  characterizes ML-randomness. This is a well-known result of \Kuc\  \cite{Kucera:85}.

\begin{prop} $Z$ is ML-random $\LR$ 
	 $Z$ is $1$-recurrent in  each  $\PPI$ class $\+ P$ with  $ 0<  p= \leb  \+ P$.  \end{prop}

\begin{proof} $\RA$:  see e.g.\ \cite[3.2.24]{Nies:book} or \cite[where?]{Downey.Hirschfeldt:book}.
	
	\n $\LA$: ML-randomness of a sequence $Z$  is preserved by   adding bits at the beginning. By the Levin-Schnorr Theorem, the $\PPI$ class $\+ P = \{ Y \colon \fa n  K(Y \uhr n) \ge n-1\}$ consists entirely of ML-randoms. So, if $Z$ is not ML-random, then no tail of $Z$ is in the $\PPI$ class $\+ P$. Further, $\leb \+ P \ge 1/2$.
\end{proof}

\begin{thm} \label{thm:ML} Let $\+ P \sub \cantor$  be a $\PPI$ class with  $ 0<  p= \leb  \+ P$. Each \ML\    random   $Z$ is multiply recurrent  in $\+ P$.
\end{thm}

\begin{proof}
As before we fix an arbitrary $k \ge 1$ in order to  show that $Z$ is    $k$-recurrent in $\+ P$.  First we prove the assertion under the additional assumption  that $1-1/k <  p$.  This  generalises \Kuc's argument in  `$\RA$' of  the proposition above,  where $k=1$ and the additional assumption $0<p$ is already satisfied.

Let $B \sub \strcantor$ be a prefix-free c.e.\ set such that $\Opcl B = \cantor - \+ P$. We may assume that $B_0 = \ES$ and for each $t>0$, if $\sss \in B_t - B_{t-1}$ then $\sssl = t$. We define a uniformly c.e.\ sequence $\seq{C^r}$ of prefix-free sets with the same property that at stage $t$ only strings of length $t$ are enumerated.
 
 For a string $\eta$ and $u \le |\eta|$, we write $(\eta)_u$ for  the string $\eta$ with the first $u$ bits removed.
 Let  $C^0$ only contain the empty string,  which is enumerated at stage $0$. Suppose  $r > 0$ and $C^{r-1}$  has been defined. Suppose $\sss$ is enumerated in $C^{r-1}$ at stage $s$ (so $\sssl =s$). For strings $\eta \succ \sss$ we search for the failure  of    $k$-recurrence in $\+ P$ that would be obtained by  taking  $s$ bits off $\eta$  for $k$ times.
At stage $t>   (k+1) s$,  for each string  $\eta $ of   length $t$  such that   $\eta \succ \sss$ and 
\[ \tag{$*$} \bigvee_{1 \le i \le k}  (\eta)_{si} \in B_{t-si}, \]
and no prefix of $\eta$ is in $C^r_{t-1}$,   put $\eta$  into $C^r$ at stage $t$.   
\begin{claim} $C^r$ is prefix-free for each $r$. \end{claim}
This holds for $r=0$. For $r> 0$ suppose that $\eta \preceq \eta'$ and both strings are in $C^r$. Let $t= |\eta|$. By inductive hypothesis the string $\eta$ was enumerated into $C^r$ via a unique  $\sss \prec \eta$, where $\sss \in C^{r-1}$. Then $\eta = \eta'$ because we chose the string in $C^{r}$  minimal under the prefix relation. This establishes the claim.
 
 By hypothesis $1> q = k \leb \Opcl B$. 
\begin{claim} \label{cl:small measure} For each $r \ge 0$ we have $\leb \Opcl{C^r}  \le q^r$. \end{claim}
This holds for $r=0$. Suppose  now that $r>0$. Let  $\sigma \in C^{r-1}$. The local measure above $\sss$ of strings $\eta $,  of a length $t$,  such that $\bigvee_{1 \le i \le k}  \eta_{si} \in B_{t-is}$ is at most $q$. The estimate follows by the prefix-freeness of $C^r$.

If   $Z$ is not $k$-recurrent  in $\+ P$, then $Z \in \Opcl{C^r}$ for each $r$, so $Z$ is not ML-random. 

We now remove the additional assumption that $1-1/k <  p$.  We define the sets $C^r$ as before.  Note that any string in $C^r$ has length at least $r$. 
  Everything will work except for Claim~\ref{cl:small measure}: if $\leb  \Opcl B\ge 1/k$ then $\leb \Opcl{C^r} $ could   be $1$. To remedy this,  we choose a finite set  $D \sub  B$ such that the set $\wt B = B- D$ satisfies $\leb \Opcl {\wt B} < 1/k$.  
Let $N = \max \{\sssl \colon \sss \in D\}$. We modify the argument of Prop.\ \ref{prop: Kurtz}, where  the clopen set  $\+ P$ there  now becomes  $\cantor - \Opcl D$.  

Let $C = \bigcup_r C^r$. Let $G_m$ be the set of prefix-minimal strings $\eta $  such that  $\eta  \in C$,  and there exist $m$ many $s> N$   as follows. \bi \item  $\eta \uhr s \in C$, and \item for some $i$ with  $1 \le i \le k$, $\eta \uhr{ [ si, s(i+1))}$ extends a string in $D$.  \ei 
(Informally speaking, if there are arbitrarily long such  strings along $Z$, then   the attempted test $\Opcl{C^r}$ might  not work, because the relevant ``block"  $\eta \uhr{ [ si, s(i+1))}$ may  extend a string  in $D$,  rather than one  in $\wt B$.)

The sets $G_m $ are uniformly $\SI 1$. By  choice of $N$ and independence, as in the proof of Prop.\ \ref{prop: Kurtz} we have $\leb \Opcl{G_{m+1}} \le (1- v^k) \leb G_m$, where $v = \leb (\cantor - \Opcl D)$. If  $Z$ is ML-random we can choose a least $m^*$ be such that $Z \not \in \Opcl{G_{m^*}}$.  

Note that $m^*>0$ since $G_0 = \{\ES\}$. So choose $\rho \prec Z$ such that $\rho \in G_{m^*-1}$. Then 
 $\rho\in C^r$ for some  $r$, and  no $\tau$ with  $\rho \preceq \tau \prec Z$ is in $G_{m^*}$. 
 
 We define a ML-test that succeeds on  $Z$. Let $\wt C^r = C^r$. Suppose  $u >  r$ and $\wt C^{u-1}$ has been defined. For each  $\sss \in \wt C^{u-1}$,     put into   $\wt C^{u} $ all the  strings $\eta \succ \sss$ in $C^{u} $ so that    ($*$) can be strengthened to  $ \bigvee_{1 \le i \le k}  (\eta)_{si} \in \wt B_{t-is}$, where $s = \sssl$. 
 
 Let $q = k \leb \Opcl {\wt B}$.  Note that  $\leb \Opcl {\wt C^u} \le q^u$ as before. By the choice of $m^*$ we have  $Z \in \bigcap_{u\ge r} \Opcl {\wt C^u}$, so since $q   < 1$,  an appropriate refinement of the sequence of open sets $\seq {\Opcl {\wt C^u}}\sN u$ shows   $Z$ is not ML-random.

\end{proof}

\subsection{Towards the general case}
\subsubsection{Recurrence for $k$   shift operators} 
  The probability space under consideration is  now  $\+ X = \{0,1\}^{\NN^{\normalsize k}}$ with the product measure.
For $1\le i \le k$, the operator $T_i \colon \+ X \to \+ X$ takes one   ``face" of bits off in direction  $i$. That is, for $Z \in \+ X$,  \bc $T_i(Z)(u_1, \ldots, u_k) = Z(u_1, \ldots, u_i+1, \ldots, u_k)$. \ec
$Z$ is \emph{recurrent} in a class $\+ P \sub \+ X$ if $ [Z \in \bigcap_{  i \le k} T^{-n}_i ( \+ P) $ for some $n$.

Algorithmic randomness notions for points in $\+ X$ can be  defined via  the  effective measure preserving isomorphism $\+ X \to \cantor$ given by a computable bijection $\NN^k \to \NN$. Modifying  the methods above,  we show the following. 

\begin{thm} Let $\+ P \sub \+ X$  be a $\PPI$ class with  $ 0< p=\leb  \+ P$.  Let $Z \in \+ X$.

If $Z$ is (a) Kurtz (b) Schnorr (c) ML-random, 
 then $Z$ is recurrent in $\+ P$    

in case (a) $\+ P$ is clopen  (b) $\leb  \+ P$ is computable (c) for any $\+ P$.\end{thm}
%%%%%%%%%%%%%%%
\begin{proof} For the duration of this proof, by an \emph{array} we mean a map $\sss \colon \{0, \ldots, n-1\}^k \to \{0,1\}$. We call $n$ the size of $\sss$ and write $n= \sssl$. The letters $\sss, \tau, \rho, \eta $ now denote arrays. For $s \le n$ and $i \le k$ let $(\sss)_{i,s}$ be the array $\tau$ of size $n-s$ such that   \bc $\tau (u_1, \ldots, u_k) = \sss (u_1, \ldots, u_i+s, \ldots, u_k)$ \ec
 for $u_1, \ldots, u_k \le n-s$.  This operation removes $s$ faces in direction $i$, and then cuts the opposite faces in the remaining directions  in order to obtain an array.
For a  set $S$ of arrays we define $\Opcl S = \{ Y  \in \+ X \colon \, \ex \sss \in S \, [\sss \prec Y]\}$ where the ``prefix"  relation $\prec$ is defined as expected.   

Suppose  that  $Z$ is not $k$-recurrent in $\+ P$ for some $k \ge 1$.

\n (a). As in Prop.\ \ref{prop: Kurtz}  we define a null $\PPI$ class  $\+ Q \sub \+ X$ containing $Z$. Let $n_1$ be least such that $\+ P = \Opcl F$ for some set of arrays of that all have size $n_1$. Let 
\[ \+ Q = \bigcap_{r \ge 1} \{ Y \colon   \bigvee_{1 \le i \le k} T_i^{rn}(Y)  \not \in \+ P\}. \]
 By the choice   of $n_1$  the conditions in  the same disjunction are independent, so  we have \[ \leb (\bigvee_{1 \le i \le k} T_i^{rn}(Y)  \not \in \+ P ) =  1 - p ^k < 1.\] The  $\PPI$ class $\+ Q$ is the independent intersection of  such classes  indexed by~$r$. Therefore $\+ Q$ is null. By hypothesis $Z \in \+ Q$. So $Z$ is not weakly random.
 
 \n (b).  We could modify the previous argument. However, this also follows by  the general fact in Remark~\ref{Rem:Rute} below. 
 
 \n (c). The argument is   similar to the proof of Theorem~\ref{thm:ML} above. The definition of the c.e.\ set $B$ and its enumeration are as before, except that  each string of length $n$ is now an array of size~$n$. In particular, an array enumerated at a stage $s$ has size $s$. 
 
 Let  $C^0$ only contain the empty array,  which is enumerated at stage $0$. Suppose  $r > 0$ and $C^{r-1}$  has been defined. Suppose $\sss$ is enumerated in $C^{r-1}$ at stage $s$ (so $\sssl =s$). 
 
 At a stage $t >  2s$,  for each array  $\eta $ of   size $t$  such that   $\eta \succ \sss$ and 
\[ \tag{$*$} \bigvee_{1 \le i \le k}  (\eta)_{i,s}  \in   B_{t-s}, \]
and no array that is a prefix of $\eta$ is in $C^r_{t-1}$,   put $\eta$  into $C^r$ at stage $t$.  
 As before one checks that   $C^r$ is prefix-free for each $r$.  

Choose a finite set of arrays   $D \sub  B$ such that the set $\wt B = B- D$ satisfies $\leb \Opcl {\wt B} < 1/k$.  
Let $N = \max \{\sssl \colon \sss \in D\}$.  
Let $C = \bigcup_r C^r$. Let $G_m$ be the set of prefix-minimal arrays $\eta  $  such that  $\eta  \in C$, and  there exist $m$ many $s> N$   as follows. 
%%%%
\bi \item  $\eta \uhr {\{ 0, \ldots, s-1\}^k} \in C$, and \item for some $i$ with  $1 \le i \le k$, $(\eta)_{i,s}$ extends an array in $D$.  \ei 
%%%%
 The sets $G_m $ are uniformly $\SI 1$. By  choice of $N$ and independence  $\leb \Opcl{G_{m+1}} \le (1- t^k) \leb G_m$, where $t = \leb (\cantor - \Opcl D)$. If  $Z$ is ML-random we can choose a least $m^*$ be such that $Z \not \in \Opcl{G_{m^*}}$, and  $m^*>0$ since $G_0 = \{\ES\}$. So choose $\eta \prec Z$ such that $\eta \in G_{m^*-1}$. Then 
 $\eta \in C^r$ for some  $r$, and  no $\tau$ with  $\eta \preceq \tau \prec Z$ is in $G_{m^*}$. 
 
% For each $u > r$ let  $\wt C^u $ be the set of arrays in $C^u$ extending $\eta$  such that ($*$) can be strengthened to   $\bigvee_{1 \le i \le k}  (\eta)_{i,s}  \in   \wt B_{t-s}$, for each $s$ with $r \le s < u$. Since   $q = k \leb \Opcl {\wt B} < 1$,  and  $\leb \Opcl {\wt C^u} \le q^u$, by the choice of $m^*$ we have  $Z \in \bigcap_{u>r} \Opcl {\wt C^u}$, so $Z$ is not ML-random.

  Let $\wt C^r = C^r$. Suppose  $u >  r$ and $\wt C^{u-1}$ has been defined. For each  $\sss \in \wt C^{u-1}$,      put into   $\wt C^{u} $ all the  arrays  $\eta \succ \sss$ in $C^{u} $ so that    ($*$) can be strengthened to   $\bigvee_{1 \le i \le k}  (\eta)_{i,s}  \in   \wt B_{t-s}$, where $s = \sssl$ and $t = |\eta|$. 
 
 Let $q = k \leb \Opcl {\wt B}$.  Then  $\leb \Opcl {\wt C^u} \le q^u$ as before. By the choice of $m^*$ we have  $Z \in \bigcap_{u\ge r} \Opcl {\wt C^u}$, so since $q   < 1$,    $Z$ is not ML-random.
\end{proof} 

\subsubsection{The  putative full result} It is likely that a multiple recurrence theorem holds in greater generality. For background on computable  probability spaces and how to define randomness notions for points in them, see e.g.\ \cite{Gacs.Hoyrup:11}.

%(In the absence of ergodicity, we have to require that $z \in \+ P$.)  

\begin{conjecture} Let $(X, \mu) $ be a computable probability space. Let $T_1, \ldots, T_k$ be computable measure preserving transformations that commute pairwise. Let $\+ P $ be a $\PPI$ class with $\mu P> 0$. 

If $z \in \+ P$ is ML-random then $\ex n  [ z \in \bigcap_{  i \le k} T^{-n}_i ( \+ P)]$. 
\end{conjecture}

\begin{remark} \label{Rem:Rute} {\rm   Let $U_n $ be the open set $ \{x \colon \,   x \not \in \bigcap_{i\le k} T^{-n}_i ( \+ P)\}$. Then $\mu ( \+ P \cap \bigcap_n U_n)=0$ by the classic multiple recurrence  theorem in the version of Cor.\ \ref{prop:FurMRT3}. Since $\+ P \cap \bigcap_n U_n$  is $\PI 2$, weak 2-randomness of $z$ suffices for the $k$-recurrence.

  Jason Rute has pointed out that if $X$ is Cantor space and $\mu \+ P$ is computable, then $\ex n [ z \in \bigcap_{  i \le k} T^{-n}_i ( \+ P)]$  for every Schnorr random $z \in \+ P$.  For in this case $\mu  \widehat U_n$ is uniformly computable where $\widehat U_n = \bigcap_{i< n}U_i$. Let $\+ P = \bigcap_n \+ P_n$ where the $\+ P_n$ are clopen sets computed uniformly in $n$. Let $G_n = \+ P_n \cap \widehat U_n$. Then $G_n$ is uniformly $\SI 1$ and $\mu (G_n)$ is   uniformly computable. Refining the sequence $\seq{G_n}$ we obtain a Schnorr test capturing $z$. }  \end{remark}

For general $\+ P$'s, an interesting first case would be when $T= S^i$ where $S$ is the  rotation of the unit circle of the form $z \to z e^{2 \pi i \alpha}$ for irrational computable $\alpha$. Such an $S$ is  ergodic, but not weakly mixing.

%\hl{Can someone try this for Kurtz and elementary open sets?}

\section[Degrees of halves of left-c.e.\ randoms]{Greenberg, Turetsky and Westrick: \\ Degrees of halves of left-c.e.\ randoms}

For $\alpha$ a real, let $\pi_0(\alpha)$ denote the real made from the even bits of $\alpha$'s binary expansion, and $\pi_1(\alpha)$ the real made from the odd bits.  Thus $\alpha = \pi_0(\alpha) \oplus \pi_1(\alpha)$.  A nonempty subset of Greenberg, Miller, Nies and Turetsky wondered the following:  If $\alpha$ and $\beta$ are left-c.e.\ random reals, must $\{ \text{deg}(\pi_0(\alpha)), \text{deg}(\pi_1(\alpha))\} = \{\text{deg}(\pi_0(\beta)), \text{deg}(\pi_1(\beta))\}$.  Note that these are unordered pairs, so the question is whether $\pi_0(\beta)$ must have the same degree as one of $\pi_0(\alpha)$ or $\pi_1(\alpha)$, and $\pi_1(\beta)$ the same degree as the other.

Greenberg, Turetsky and Westrick have answered the question in the negative, via the following:

\begin{lemma}
If $\alpha$ is random, and $d = 2k+1$ is an odd integer with $d \geq 3$, then $\pi_i(\alpha) \oplus \pi_j(\alpha/d)$ can derandomize $\pi_{1-i}(\alpha)$ for $i, j < 2$.
\end{lemma}

Presumably a stronger result is possible: relative to $\pi_j(\alpha/d)$, $\pi_i(\alpha)$ should have effective Hausdorff dimension at most 1/2, with no assumptions on $\alpha$.

\begin{proof}
The proof is by long division.  For $\sigma \in 2^{<\omega}$, let $\text{int}(\sigma)$ be the integer denoted by $\sigma$ as a big-endian binary representation.  That is, $\text{int}(\sigma) = \sum_{\ell < |\sigma|} \sigma(\ell)\cdot 2^{|\sigma|-\ell - 1}$.  We will be making reference to $\text{int}(\alpha\!\!\upharpoonright_n)$.  We assume $\alpha$ is a real between 0 and 1, so we identify it with the infinite sequence of 0s and 1s in its binary expansion to the right of the radix symbol.  When passing from $\alpha$ to $\alpha\!\!\upharpoonright_n$ to $\text{int}(\alpha\!\!\upharpoonright_n)$, we drop the radix symbol to obtain an integer.

Recall the long division algorithm.  When performing the division $\text{int}(\alpha\!\!\upharpoonright_n) \div d$, the quotient is $\text{int}((\alpha/d)\!\!\upharpoonright_n)$, and there is some remainder $r < d$.  Further, if $b_0$ and $b_1$ are the next two bits of $\alpha$ and $c_0$ and $c_1$ are the next two bits of $\alpha/d$, so that $\alpha\!\!\upharpoonright_{n+2} = (\alpha\!\!\upharpoonright_n)*b_0b_1$ and $(\alpha/d)\!\!\upharpoonright_{n+2}=((\alpha/d)\!\!\upharpoonright_n)*c_0c_1$, then the quotient of $(4r +2b_0+b_1) \div d$ is $\text{int}(c_0c_1) = 2c_0 + c_1$, again with some remainder.  This is simply the ``carry'' procedure of long division, performed over two bits at a time rather than a single bit.

Recall that $d = 2k+1$.  Since $\alpha$ is random, there are infinitely many $n$ such that $\text{int}(\alpha\!\!\upharpoonright_n)\div d$ has remainder $k$.  In fact, there are infinitely many such $n$ which are even and infinitely many which are odd.  For such an $n$, let $b_0, b_1, c_0$ and $c_1$ be as above.  Observe that $4k + 2b_0 + b_1 \ge 2(2k+1)$ iff $b_0 = 1$.  So $c_0 = 1$ iff $b_0 = 1$.  Also, if $b_0 = 1$, then $4k+2b_0+b_1 < 2(2k+1) + (2k+1)$, since $d \ge 3$ and thus $k \ge 1$.  On the other hand, if $b_0 = 0$, then $4k+2b_0 + b_1 > 2k+1$, again since $k \ge 1$.  Thus $c_1 = 1$ iff $b_0 = 0$.

We now describe a martingale computable from $\pi_i(\alpha) \oplus \pi_j(\alpha/d)$ that succeeds on $\pi_{1-i}(\alpha)$.  By reading bits of $\pi_{1-i}(\alpha)$ and combining this with $\pi_i(\alpha)$ from its oracle, our martingale can obtain initial segments of $\alpha$.  For each $\alpha\!\!\upharpoonright_n$ it computes the remainder of $\text{int}(\alpha\!\!\upharpoonright_n)\div d$.  When it sees that the remainder is $k$, and $n$ is such that the next bit of $\alpha$ is from $\pi_{1-i}(\alpha)$ ($n$ is even or odd, as appropriate), it is ready to bet on the next bit of $\pi_{1-i}(\alpha)$.  By the calculations above, both the $n$th bit $c_0$ and the $n+1$st bit $c_1$ of $\alpha/d$ determine the $n$th bit of $\alpha$.  From the $\pi_j(\alpha/d)$ in the oracle, our martingale knows one of these bits, and so it knows the next bit of $\pi_{1-i}(\alpha)$.  So it bets all its money on this bit.  Our arguments above show that this martingale succeeds on $\pi_{1-i}(\alpha)$.
\end{proof}

Now, suppose $\alpha$ is left-c.e.\ and random.  Then $\beta = \alpha/3$ is also left-c.e.\ and random.  However, neither of $\pi_0(\beta)$ or $\pi_1(\beta)$ can be computable from $\pi_0(\alpha)$, as the lemma would then say that $\pi_0(\alpha)$ could derandomize $\pi_1(\alpha)$, contrary to van Lambalgen's theorem.

\part{Randomness and analysis}

\newtheorem{problem}{Problem}%Needed for problem statements.
\section{Rute: Research directions and open problems for ARA 2014 Japan}

(By Jason Rute, Pennsylvania State University.) This is a revised version on a list of research directions and open problems for Analysis, Randomness, and Applications (ARA) 2014 in Japan%
\footnote{
\url{http://kenshi.miyabe.name/ara2014/}}%
.  (The author was not in attendance and sent these notes in absentia.)

\subsection{Which randomness notions are natural?}

\subsubsection{Determine the natural randomness notions}

As already pointed out by Schnorr \cite{Schnorr:75} and even, perhaps,  \ML\  \cite{MartinLof:68}, there is not just one natural randomness notion.
There are at least two---ML-randomness and Schnorr-randomness---and probably more---computable
randomness, $n$-randomness, weak $n$-randomness, and higher randomness
notions.  Here are some problems attempting to systemically understand the
collection of randomness notions.
\begin{problem}
Axiomatize randomness. 
\end{problem}
Axioms will probably include versions of preservation of randomness,
no randomness from nothing, and van Lambalgen's theorem---or possibly an entirely different approach.  Van Lambalgen \cite{vLamb:90} attempted an axiomatization.  Recently, both Rute and Simpson%
\footnote{See \url{http://homepages.inf.ed.ac.uk/als/Talks/leeds-istr15.pdf}.} %
have also been working separately on axiomatizations.

\begin{problem}
Characterize the nice randomness notions.
\end{problem}
This is just another way of stating the last problem. I suspect that
the ``nice randomness'' notions will come out to be those equal
to Schnorr randomness relative to a class of oracles. 
(For example, it is already known $x$ is ML random is and only if it is equal to computable randomness relative to a PA degree.  This follows from the paper of Brattka, Miller, and Nies \cite{Brattka.Miller.ea:16}.)

\begin{problem}
Show that the ``weird but natural'' randomness notions are more
natural randomness notions in disguise.
\end{problem}
This has already been done for strong $s$-randomness and $s$-energy
randomness. They are equivalent to randomness for capacities (or equivalently
randomness for effectively compact classes of measures). I conjecture UD-randomness
is really Schnorr randomness for a class of measures, and the differentiability
points of all Lipschitz functions of type $\mathbb{R}^{n}\rightarrow\mathbb{R}^{m}$
($m<n$) are the computable randoms for a certain class of measures.

\begin{problem}%
\footnote{This problem is a major part of his Alex Galicki's ongoing PhD thesis. I communicated my conjectures to him in Spring 2014 suggesting that he work this out.} Let $m\geq1$ and $n\geq 1$.  Characterize the points of differentiability of all computable Lipschitz functions of type $\mathbb{R}^{n}\rightarrow\mathbb{R}^{m}$.  

More specifically, is the following conjecture true? If $x\in \mathbb{R}^n$ is a point of differentiability of all computable Lipschitz functions of type $\mathbb{R}^{n}\rightarrow\mathbb{R}^{m}$ iff
\begin{enumerate} 
\item ($m\geq n$) $x$ is computably random w.r.t.\  $\mathbb{R}^n$ with the Lebesgue measure.\footnote{%
Compare with Preiss and Speight \cite[Paragraph above Thm 1.1]{Preiss.Speight:15}.}
\item ($m = 1 < n$) $x$ is computably random w.r.t.\ some measure $\mu$ (possibly non-computable) with $n$ independent Alberti's representations (equivalently, $(\mathbb{R}^n,\mu)$ is a Lipschitz differentiability space).  (Possibly some smaller subclass of measures $\mu$ is sufficient.)\footnote{%
Compare with Alberti, Csörnyei, and Preiss \cite{Alberti.Csornyei.ea:10} (also the slide ``Differentiability and singular measures'' in \url{https://www.ljll.math.upmc.fr/~lemenant/GMT/preiss.pdf}) and with Bates \cite{Bate:15} which generalizes Alberti, Csörnyei, and Preiss.}
\item ($m < n$) $x$ is computably random w.r.t.\ some measure $\mu$ in a nice class of measures.
\end{enumerate}

\end{problem}

On a related note, we need to be more careful what we call randomness
notions. Randomness notions are associated to a measure (or class
of measures). Many things we call randomness notions are not randomness
notions on the Lebesgue measure. Here I am defining a randomness notion
as a class of ``computable tests'' $T$, each of which is identified
with a null set, and this class of tests can also be relativized to
oracles. A ``randomness notion'' should only be considered a randomness
notion \emph{for Lebesgue measure} if for each Lebesgue null set
$N$, there is a test $T$ relative to an oracle such that $N\subseteq T$.
Therefore Kurtz randomness and UD randomness are not randomness notions
on the Lebesgue measure!
\begin{problem}
Understand Martin-Löf randomness for lower-semicomptuable semi-measures
better.

There are a lot of definitions out there, both of lower-semicomputable
semi-measure, and of randomness for such an object. This needs to
be worked out. We actually implicitly use randomness for semi-measures
already in many of our proofs. (For example, the proof of Miller and
Yu's theorem that if $X$ is MLR, $X\leq_{T}Y$, and $Y$ is MLR relative
to $Z$, then $X$ is MLR relative to $Z$.)
\end{problem}

\subsection{Studying the structure of the ``random degrees''}

The LR degrees, while a good start, are too coarse for this purpose.
The Turing degrees have very little connection to randomness. The
truth-table degrees are better; the set $\mathsf{MLR}_{\textnormal{comp}}=\{x\mid x\in\mathsf{MLR}_{\mu},\mu\text{ computable}\}$
is downward closed in the truth-table degrees (and the same for $\mathsf{SR}_{\textnormal{comp}}$).
However, we can do better. I have a number of ideas for this project
(see my slides from the 2013 Nancy ARA meeting at \url{http://www.personal.psu.edu/jmr71/talks/rute_2013_ARA.pdf}). Here are some related
problems.
\begin{problem}\label{Prob6}
(Bienvenu and Porter) If $T$ is a computable measure-preserving map
and $y$ is Schnorr random, is there some Schnorr random $x$ such
that $T(x)=y$?
\end{problem}
(Update: Problem \ref{Prob6} has been since answered negatively by Rute.)

For this next problem, if $\mu$ is a computable measure on $2^{\mathbb{N}}\times2^{\mathbb{N}}$,
then there is a kernel measure $\mu(\cdot\mid x)$ defined by $\mu(\sigma\mid x)=\lim_{n}\frac{\mu([x\upharpoonright n]\times[\sigma])}{\mu([x\upharpoonright n]\times2^{\mathbb{N}})}$.
If $x$ is computably random on the marginal measure $\mu_{1}$ (where
$\mu_{1}(\sigma)=\mu([\sigma]\times2^{\mathbb{N}})$) then $\mu(\cdot\mid x)$
is a measure, although it may not be computable from $x$. 
\begin{problem}
(Shen, Takahashi, Bauwens) Suppose $\mu$ is computable measure on $2^{\mathbb{N}}\times2^{\mathbb{N}}$
and $x$ is Martin-Lof random on $\mu_{1}$.  Characterize the
set $\{y\mid(x,y)\in\mathsf{MLR}_{\mu}\}$ in terms of $x$ and $\mu(\cdot\mid x)$.
Is this possible without knowing $\mu$?
\end{problem}
Takahashi conjectured that $(x,y)\in\mathsf{MLR}_{\mu}$ if and only
if $x\in\mathsf{MLR}_{\mu_{1}}$ and $y$ is blind random for $\mu(\cdot\mid x)$
relative to $x$. But this is wrong. Bauwens gave a counterexample
(communicated to Shen and me).

\begin{problem}
If $x$ is computably random, and $y$ is computably random relative
to $x$, then is $(x,y)$ necessarily computably random?
\end{problem}

\subsection{Extending results about MLR to SR}

Most of the results about MLR and analysis extend to SR. This includes
almost all the work on basic measure theory, almost all the work on
Brownian motion, probably much of the work on effective dimension,
and a good deal of the results about ergodic theory. (Miyabe has also
done a great job extending the results on LR degrees.) The tools to
do this are almost mature. They include
\begin{enumerate}
\item A good understanding of the connection between measurable functions
and Schnorr randomness, including Schnorr layerwise computability
and preservation of Schnorr randomness. (Miyabe \cite{Miyabe:13}; Pathak, Rojas, Simpson \cite{Pathak.Rojas.ea:12}; Rute \cite{Rute:13})
\item A good understanding of the connection between effective rates of
convergence and Schnorr randomness. (Gács, Hoyrup, Rojas \cite{Gacs.Hoyrup:11}; Galatolo,
Hoyrup, Rojas \cite{Galatolo.Hoyrup.ea:10}; Rute \cite{Rute:13})
\item A good-enough version of no-randomness-from-nothing. (Rute%
\footnote{In preparation. See slides at \url{http://www.personal.psu.edu/jmr71/talks/rute_2013_ARA.pdf}.}%
)
\item Van Lambalgen's theorem holds for Schnorr randomness under uniform
reducibility. (Miyabe \cite{Miyabe:11}; Miyabe, Rute \cite{Miyabe.Rute:13})
\item Schnorr randomness can be defined on non-computable measures. (Rute%
\footnote{In
preparation. See slides at \url{http://www.personal.psu.edu/jmr71/talks/rute_2014_jmm.pdf}.}%
)
\end{enumerate}
Now the problem is to actually do the work. (Update: Rute has slowly been compiling a bunch of  facts about Schnorr randomness and Brownian motion.)

\subsection{Gathering the information we have on randomness and analysis}

There has been a lot of work done in the last decade on randomness
and analysis, but it is spread throughout a bunch of papers. There
are not any books yet on the subject. (The closest approximations are
Gács's lecture notes \cite{Gacs:05}; the article by Bienvenu, Gács, Hoyrup, Rojas,
and Shen \cite{Bienvenu.Gacs..ea:11}; and the second part of Rute's thesis \cite{Rute:13}.)
It would be nice to gather this material better.  (Update: Rute is currently writing a survey article on randomness and analysis.)

\part{Computability theory and its connections to other areas}

\section[Patey: effectively bi-immune sets]{Patey: effectively bi-immune sets and  \\ computably bounded DNC functions}

The following section has been written by Ludovic Patey in March 2015.

\begin{definition}
A function $f$ is $h$-bounded for some computable $h$ if $f(e) \leq h(e)$ for all $e$.
A set $A = \{x_0 < x_1 < \dots \}$ is $h$-bounded if its principal function ($p_A : n \mapsto x_n$) is $h$-bounded.
A function $f$ is \emph{fixed-point free} if $W_{f(e)} \neq W_e$ for all $e$.
A function $f$ is \emph{diagonally non-computable} (DNC) if $f(e) \neq \Phi_e(e)$ for all $e$.
A function $f(\cdot, \cdot)$ is \emph{escaping} if $|W_e| \leq n \rightarrow f(e,n) \not \in W_e$ for all $e$.
\end{definition}

The degrees of DNC functions are known to be equivalent to the degrees of effectively immune sets.
Jockusch and Lewis~\cite{Jockusch2013Diagonally} proved that one can compute a bi-immune set from a DNC function,
and asked whether every DNC function computes an effectively bi-immune set.
Beros~\cite{Beros2013DNC} answered negatively with an elaborate construction.
We prove that every effectively bi-immune set computes a computably bounded DNC function.
This fact is sufficient to answer Jockusch and Lewis question, since it is known
that there exists a DNC function computing no computably bounded DNC function~\cite{Kjos.Lempp.ea:04}.

\begin{lemma}\label{lem:coimmune-bound}
Every effectively co-immune set is computably bounded.
\end{lemma}
\begin{proof}
Let $X$ be an $h$-co-immune set for some computable function~$h$.
We first build a computable function $f$ such that $f(n)$ bounds the $n$th element of $X$.
Let $j$ be a computable function which on input $n$ returns the value $n+h(e)$
for some $e$ such that $W_e = [n, n+h(e))$. Such a function is computable by uniformity
of Kleene's recursion theorem. Consider the following $f$ defined by
$$
f(0) = 0 \mbox{ and } f(n+1) = j(f(n))
$$
We prove by induction over $n > 0$ that the $n$th element
of $X$ is smaller than $f(n)$. Suppose that the first $n$ elements
of $X$ are smaller than $f(n)$.
$f(n+1) = j(f(n)) = f(n) + h(e)$ for some $e$ such that $W_e = [f(n), f(n)+h(e))$.
$|W_e| \geq h(e)$ so by $h$-co-immunity of $X$, $W_e \not \subset \overline{X}$
and so there must be an element of $X$ in the interval $[f(n), f(n)+h(e))$
and therefore the $(n+1)$th element of $X$ is smaller than $f(n)+h(e) = f(n+1)$.
\end{proof}

\begin{theorem}\label{thm:bounded-equivalence}
Fix a set $X$. The following are equivalent:
\begin{itemize}
	\item[(i)] $X$ computes a computably bounded effectively immune set.
	\item[(ii)] $X$ computes a computably bounded fixed-point free function.
	\item[(iii)] $X$ computes a computably bounded DNC function.
	\item[(iv)] $X$ computes a computably bounded escaping function.
\end{itemize}
\end{theorem}
\begin{proof}
This is exactly the standard proof of equivalent between effectively immune sets, fixed-point free functions,
DNC functions and escaping function, noticing that we can transmit the computable bound.
\end{proof}

\begin{corollary}
Every effectively bi-immune set computes a computably bounded DNC function.
\end{corollary}

\begin{corollary}[Beros]
There exists a DNC function computing no effectively bi-immune set.
\end{corollary}

\section{Up to $2 \cdot 2 \cdot 2^{\aleph_0}$ cardinal invariants,  and their counterparts in computability theory}

J\"org Brendle and Andr\'e Nies met in Kobe, Japan. They discussed $2\cdot 2$ families of cardinal invariants parameterised by reals. There are two families, each with a duals. The first family is defined in terms of a  bound $h$  on functions in 
${}^\omega\omega$. The second is defined by a real parameter as a bound on asymptotic density.  All characteristics have analogs in computability theory. 

Separation is unknown in many interesting cases, in both areas. One particular separation in computability would solve the Gamma question posed in  \cite{Andrews.etal:16}.

We don't claim originality for all the notions. A lot of them are at least implicit in  previous work.  Set theory: Goldstern and  Shelah; Kihara; Computability theory: Andrews et al.;  Monin and Nies.

\subsection{Background}

We follow Brendle et al.\ \cite{Brendle.Brooke.ea:14}, some of  which in turn relies on work of Rupprecht  \cite{Rupprecht:10} and his thesis \cite{Rupprecht:thesis}.

Let  $R \subseteq  X \times Y$ be a relation between spaces $X,Y$ (such as Baire space) satisfying $\forall x \; \exists y \;
(x R y)$ and $\forall y \; \exists x \; \neg (x R y)$. Let $S = \{  \langle y,x \rangle \in Y \times X \colon \neg xR y\}$.  

\begin{definition} \label{df: bd} We write

\[ \frd(R) = \min\{|G|:G\subseteq Y \land \, \forall x \in X \,
\exists y \in  G   \, xR y\}.\]

\[ \frb(R) = \frd (S) =  \min\{|F|:F\subseteq X \land \, \forall y\in Y
\exists x \in F   \, \neg  xR y\}.\]
\end{definition}

\subsection{The parameterised  families of relations} \label{ss:hbrelations}

We will study $\frd(R)$ and $\frb(R)$ for  two types of relations $R$.

\

\n {\it 1.}   Let $h \colon \omega \to \omega $  (usually unbounded).   Define  for $  x \in {}^\omega\omega$ and 

\n $  y \in \Pi_n \{0, \ldots, h(n)-1\}$,   

\[ x \neq^*_h  y \LR \fa^\infty n \,  [ x(n) \neq y(n)]. \]

\

\n {\it 2.} Let $ 0 \le p \le 1/2$. Define,   for $x,y \in {}^\omega 2$

\[ x \sim_p y \LR   \underline \rho (x \lra y)  >p, \]
where $x \lra y$ is  the set of $n$ such that $x(n) = y(n)$, and $\underline \rho$ denotes the lower density: $\ul \rho (z) = \liminf_n |z \cap n|/n$.

\subsection{The cardinal characteristics}
\n It will be helpful to  express Definition~\ref{df: bd} for these relations   in words.

\

\n $\frd(\neq^*_h)$ is the least size of a set $G$ of $h$-bounded functions so that for each function $x$ there is a function $y$ in   $G$  such that    $\fa^\infty n [ x(n) \neq y(n)]$.  (Of course it suffices to require this for $h$-bounded $x$.)

\

\n  $\frb(\neq^*_h)$ is the least size of a set   $F$ of  functions such that for each $h$-bounded function $y$, there is a function $x$ in $F$  such that $\ex^\infty n \, x(n) = y(n)$.   (Of course we can require that each function in $F$ is $h$-bounded.)

The characteristics  $\frb(\neq^*_h)$ have been  studied in \cite{Kamo.Osuga:14} within a more general framework; the notation there is $c^\ex_{h,1}$. See Thm.\ \ref{thm:Kamo_Osuga} below.

 \

 \n $\frd(\sim_p)$ is the least size of a set $G$ of bit sequences  so that for each sequence $x$ there is a sequence  $y$ in  $G$   so that $\ul \rho (x \lra y) >p$. 
 
\

\n   $\frb(\sim_p)$ is the least size of a set  $F$ of  bit sequences such that for each bit sequence $y$, there is a sequence $x$ in $F$  such that $\ul \rho (x \lra y) \le p$. 

\subsection{Basic facts about the  cardinal characteristics}
We first record the  obvious monotonicity properties.
\begin{fact} \

\bi \item[(i)]  $h \le^* g$ implies $ \frd(h) \ge \frd(g)$ and $\frb(h) \le \frb(g)$.  

\item[(ii)] $p \le q$ implies $ \frd(\sim_p) \le \frd(\sim_q)$ and $\frb(\sim_p) \ge \frb(\sim_q)$. \ei  \end{fact}

The following facts are somewhat less obvious.
\begin{fact} \label{fact: bd} \

\bi \item[(i)] Let $h$ be bounded. Then  {\rm  (a)}   $ \frd(\neq^*_h) = \tp{\aleph_0}$ and  {\rm  (b)} $ \frb(\neq^*_h) = 2$.
\item[(ii)] {\rm  (a)} $ \frd(\sim_{0.5}) = \tp{\aleph_0}$ and {\rm (b)} $ \frb(\sim_{0.5}) = 2$.  \ei  \end{fact} 
\begin{proof}

\n {\rm (i.a)}  missing

\n {\rm (i.b)} missing

\n {\rm (ii.a)} Suppose $G$ is as in the definition   above so that $|G|=  \frd(\sim_{0.5}) $. Define a map 
$ \Theta \colon G \to \+ P(\omega)$ by $ \Theta(y) (n) =1 $ iff at least half the bits of $y$ in the interval $I_n= [n!, (n+1)!)$ are $1$. Given  $x$ that is constant on each $I_n$, pick  $y$ such that $\ul \rho (x \lra y) >0.5$. Then $\Theta(y) =^* x$. This shows that $|G| =  \tp{\aleph_0}$.

\n {\it (ii.b)}  In the definition of $ \frb(\sim_{0.5})$, let $F = \{0^\infty, 1^ \infty \}$.   \end{proof}

 \subsection{Placement within the Cicho\'n Diagram}
 A version of  Cicho\'n's  Diagram is in  Figure~\ref{Fig:10-diagram}.
 
  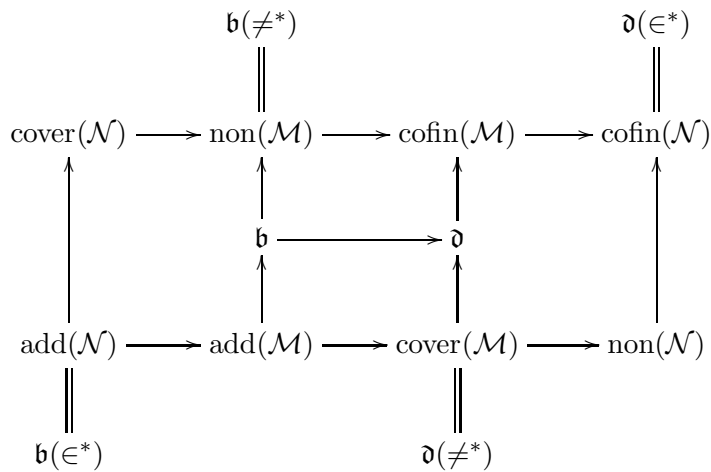
\begin{figure}[htbp]
\begin{equation*}
\xymatrix{
&\frb(\neq^*)\ar@{=}[d]&&\frd(\in^*)\ar@{=}[d]\\
\cov(\+ N)\ar[r]&\non(\+ M)\ar[r]&\cof(\+ M)\ar[r]&\cof(\+ N)\\
&\frb\ar[u]\ar[r]&\frd\ar[u]&\\
\add(\+ N)\ar[uu]\ar [r]&\add(\+ M)\ar[u]\ar[r]&\cov(\+ M)\ar[u]\ar[r]&
\non(\+ N)\ar[uu]\\
\frb(\in^*)\ar@{=}[u]&&\frd(\neq^*)\ar@{=}[u]
}
\end{equation*}
\caption{Cicho\'n's  diagram.}
\label{Fig:10-diagram}
\end{figure}
 We conjecture that one can insert the new characteristics in the right lower, and  left upper regions of the diagram.

\begin{proposition}   Let $h$ be a function. Let $0 < p < 1/2$. 

\bi \item[(i)]   $\cov (\+ M) \le \frd(\neq^*_h) \le \non( \+ N)$ and $\cov (\+ M) \le \frd(\sim_p) \le \non( \+ N)$. 
Dually, 

\item[(ii)]  $\cov (\+ N) \le \frb(\neq^*_h) \le \non( \+ M)$ and $\cov (\+ N) \le \frb(\sim_p) \le \non( \+ M)$. \ei
\end{proposition}

\begin{proof}[Partial proof]  First we settle the case of characteristics involving $\neq^*_h$.  It is trivial that $\frd(\neq^*) \le \frd(\neq^*_h)$ and $ \frb(\neq^*_h) \le  \frb(\neq^*)$. Using the equalities in the diagram, this yields  two inequalities involving  $\neq^*_h$.

Next we consider the  case of characteristics involving $ \sim_p$. 

\

\n  $\frb(\sim_p)$: For each bit sequence $x$, for almost every $y$ we have $\rho (x \lra y) = 1/2$ by law of large numbers, so  the set $\{ y \colon \, \ul \rho(x \lra y) \le p$ is null. Hence any set $F$ as in the definition  of $\frb(\sim_p)$ yields a collection of null sets of size at most $|F|$ with union $\cantor$. Hence $\cov (\+ N) \le  \frb(\sim_p)$.  

\

\n  $\frd(\sim_p)$: Let $V$ be a non-null set. For each $x$, as said the set $\{ y \colon \, \ul \rho(x \lra y) >  p$ is co-null and hence contains an element $y \in V$. Therefore $\frd(\sim_p) \le \non(\+ N)$.

\end{proof} 

\subsection{Consistency of separation of uncountably many $\frb(\neq^*_h)$}

\begin{thm}[Kamo and Osuga \cite{Kamo.Osuga:14}, Thm.\ 1] \label{thm:Kamo_Osuga} Let $\delta$ be an ordinal and let $\seq{\lambda_\alpha}_{\alpha < \delta}$ be a strictly increasing sequence of regular cardinals. Let $\kappa\ge \delta$ be a cardinal such that $\kappa =   \kappa^{< \lambda_\alpha}$ for each $\alpha < \delta$. 

There is  a forcing notion $\mathbb P$ with the  c.c.c.\ that forces: there is a sequence of functions $\seq {h_\alpha}_{\alpha < \delta}$ such that $\frb(\neq^*_{h_\alpha}) = \lambda_\alpha$ for each $\alpha$, and $\kappa = \tp{\aleph_0}$.

Moreover, if $\delta \le  \frb$ then the sequence $\seq {h_\alpha}_{\alpha < \delta}$ can be chosen in the ground model.
\end{thm}

\subsection{The corresponding highness properties in computability theory}
We will re-obtain some properties that are at least close to some well known classes. Others are new.

As before,  we follow \cite{Brendle.Brooke.ea:14}. Again, let  $R \subseteq  X \times Y$ be a relation between spaces $X,Y$, and  let $S = \{  \langle y,x \rangle \in Y \times X \colon \neg xR y\}$.
Suppose we have specified what it means for     objects $x$ in $X$,    $y$ in  $Y$  to be computable in a Turing oracle $A$. We denote this by for example $x\leT A$. In particular, for $A= \emptyset$ we have a notion of computable objects.

Let the variable $x$ range over $X$, and let $y$ range over $Y$. We define the highness properties

\[ \+ B(R) =   \{ A:  \,\exists y \leT A \, \fa x \ \text{computable} \ [xRy]     \} \]

\[ \+ D(R) =  \+ B(S) =   \{ A:  \,\exists x \leT A \, \fa y \ \text{computable} \ [\neg xRy]   \}  \]

Let $h$ be computable, and let $p \in [0,1]$. Recall the relations $\neq^*_h$ and $\sim_p$ from Subsection~\ref{ss:hbrelations}.
\n Expressing  the definitions above in words,

\

\n $\+ D(\neq^*_h)$ is the class of oracles $A$ that compute a function $x$  such that   for each computable function $y \le h$,  we have  $\ex^\infty n [ x(n) = y(n)]$.  This is called ``$h$-infinitely often equal" in \cite{Monin.Nies:15}.

\

\n  $\+ B(\neq^*_h)$ is the class of oracles $A$ that compute a function $y \le h$  such that   for each computable function $x$, we have     $\fa^\infty n \, x(n) \neq y(n)$.

 \

 \n $\+ D(\sim_p)$  is the class of oracles $A$ that compute a set $x$  such that   for each computable  set $y$, we have $\ul \rho (x \lra y) \le p$.  We note that
 
 \bc $\Gamma(A) < p \RA A \in \+ D(\sim_p) \RA \Gamma(A) \le p$. \ec 
 The right arrow cannot obviously be reversed. It could be that  $\Gamma(A) \le p$ because $\gamma(x)$ gets arbitrarily close to $p$ from above, for sets $x \leT A$.  For $p=1/2$, the reverse arrow holds by Fact~\ref{fact:BD} below. (There is a related   open question at the end of the paper Andrews et al. \cite{Andrews.etal:2013}.) 
 
\

\n   $\+ B(\sim_p)$  is the class of oracles $A$ that compute a set $y$  such that   for each computable  set $x$, we have $\ul \rho (x \lra y) > p$.  This is some kind of dual $\Gamma$ class.

\subsection{Basic facts about the  highness properties}
Again we  note obvious monotonicity properties.
\begin{fact} \

\bi \item[(i)]  $h \le^* g$ implies $ \+ D(h) \supseteq \+ D(g)$ and $\+ B(h) \sub \+ B(g)$.  

\item[(ii)] $p \le q$ implies $ \+ D(\sim_p) \sub \+ D(\sim_q)$ and $\+ B(\sim_p) \supseteq \+ B(\sim_q)$. \ei  \end{fact}

%%%%%%%%%
We proceed to the analog of Fact~\ref{fact: bd}.
\begin{fact}  \label{fact:BD} \

\bi \item[(i)] Let $h$ be bounded. Then 

\n  {\rm  (a)}   $ \+ D(\neq^*_h) =$ non-computable  and  {\rm  (b)} $ \+ B(\neq^*_h) = \ES$
\item[(ii)] {\rm  (a)}   $ \+ D(\sim_{0.5}) \LR $ non-computable and {\rm (b)} $ \+ B(\sim_{0.5}) = \ES$.  \ei  \end{fact} 
\begin{proof}

\n {\rm (i.a)}  This is nontrivial: see Monin and Nies \cite[Thm.\ IV.1]{Monin.Nies:15}.

\n {\rm (i.b)}  Trivial: take constant functions $x$ with value up to the bound on $h$.

\n {\rm (ii.a)}   $A \in  \+ D(\sim_{0.5})$ implies $\Gamma(A) \le 0.5$, so $A$ is non-computable. 

Now suppose $A$ is non-computable, and let $x(k) = A(n)$ for each $k \in I_n$ defined as in the corresponding fact above. Then for each computable $y$ we have $\ul \rho (x \lra y) \le 0.5$, else we could decide $A(n)$ for almost all $n$ by looking at the majority of values of $y$ in $I_n$.

\n {(ii.b)}    Trivial again: take a computable $x$ and its  complement. % In the definition of $ \frb(\sim_{0.5})$, let $F = \{0^\infty, 1^ \infty \}$. 
  \end{proof}

 \subsection{Placement within the Cicho\'n Diagram}
The computability theoretic Cicho\'n Diagram is given in  Figure~\ref{Fig:computability_diagram}.

\begin{figure}[htbp]
\begin{equation*}
\xymatrix{
&  \text{high or d.n.c.} \lra \+ B(\neq^*)
\ar^{\parbox{1cm}{\scriptsize \text{Ref.~\cite{Rupprecht:thesis}}}}@{=}[d]
&&\parbox{2.5cm}{\center $\+ D(\in^*) \lra \  $not\\ computably traceable}\ar^{\text{Ref.~\cite{Terwijn.Zambella:01}}}@{=}[d]\\
\parbox{1.8cm}{\center $A\ge_Ta$  Schnorr random}\ar^{\text{Ref.~\cite{Nies.Stephan.ea:05}}}[r]&
\parbox{2cm}{\center{weakly meager engulfing}}\ar[r]&
\parbox{2cm}{\center not low for weak 1-gen  \\  (i.e. hyperimmune or d.n.c.\cite{Stephan.Yu:nd})}\ar^{\text{Ref.~\cite{Kjos.ea:2005}}}[r]&
\parbox{1.8cm}{\center not low for Schnorr tests}\\
&\text{high}\ar[u]\ar[r]&
\parbox{2.3cm}{\center hyperimmune degree}\ar[u(0.45)]&\\
\parbox{1.3cm}{\center Schnorr engulfing} \ar[uu(0.85)]\ar@{=}[r]
%_{\text{Thm. \ref{thm:highmeager} and \cite{Rupprecht:10}}}
\ar^{\text{Ref.~\cite{Rupprecht:10}}}@{=}[ur]&
\parbox{1.3cm}{\center meager engulfing}\ar_{
\parbox{1cm}{\scriptsize \text{Ref.~\cite{Rupprecht:thesis}}}}@{=}[u]
\ar[r]&
\parbox{1.4cm}{\center weakly\\ 1-generic degree}\ar^{\text{Ref.~\cite{Kurtz:81}}}@{=}[u]\ar^{\text{Ref.~\cite{Rupprecht:10}}}[r]&
\parbox{1.3cm}{\center weakly Schnorr engulfing}\ar[uu(0.85)]\\
 \+ B(\in^*)
\ar^{\text{Ref.~\cite{Rupprecht:10}}}@{=}[u]&&
 \+ D(\neq^*)\ar@{=}[u]
}
\end{equation*}
\caption{The analog  of Cicho\'n's  diagram in computability.}
\label{Fig:computability_diagram}
\end{figure}
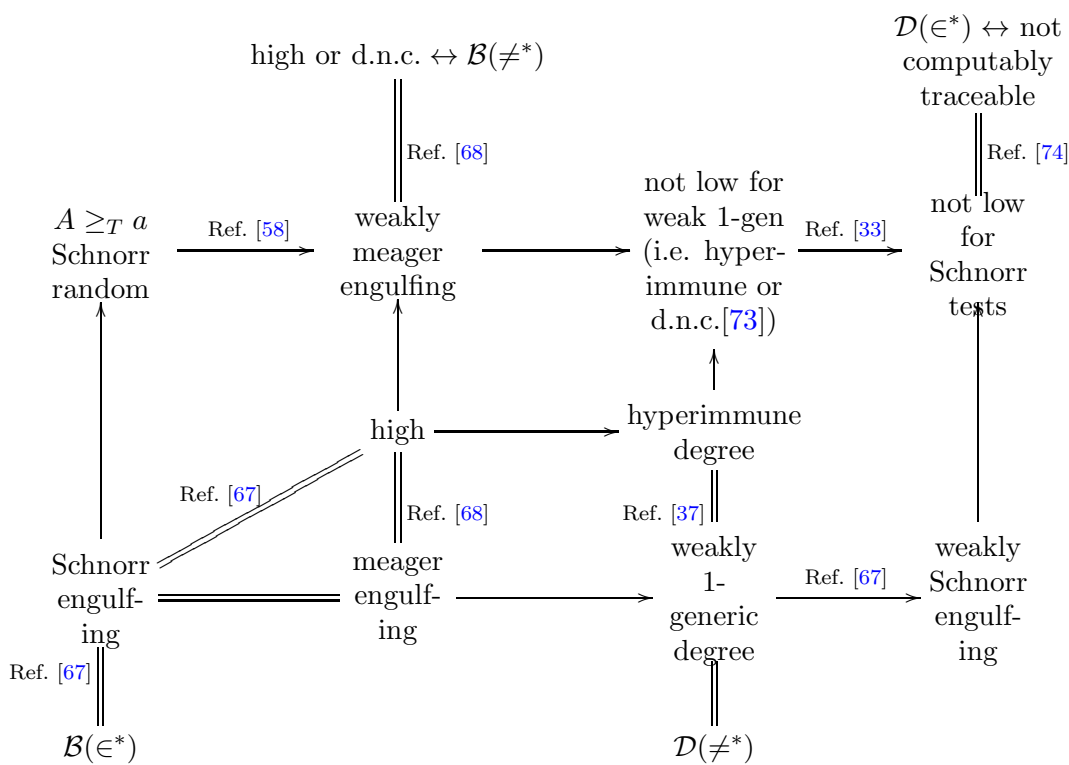

(ia) and (ib) below     is known; see  \cite{Monin.Nies:15}. We conjecture the analogs (iia) and (iib).
\begin{proposition}   Let $h$ be an order function. Let $0 < p < 1/2$ be computable.  Let $A \sub \omega$.

\bi \item[(ia)]   $A$  is of h.i.\ degree $\RA A \in  \+ D(\neq^*_h)  \RA  A$ is weakly Schnorr engulfing.

\item[(ib)]  $A$  is of h.i.\ degree $\RA A \in  \+ D(\sim_p)  \RA  A$ is weakly Schnorr engulfing.

\item[(iia)]   $A$ computes a Schnorr random $\RA A \in  \+ B(\neq^*_h)  \RA  A$ is high or d.n.c. 

\item [(iib)] $A$ computes a Schnorr random $\RA A \in  \+ B(\sim_p)  \RA  A$ is high or d.n.c. 
 \ei
\end{proposition}

\subsection{Relating the two types of highness properties}

\begin{thm}[Monin and Nies \cite{Monin.Nies:15}, Thm III.4 restated] Let $h$ be sufficiently fast growing in that  $\fa n \,  h(n) \ge 2^{(d^n)}$, for some $d>1$. Then,  \bc 
$ A \in  \+ D (\neq^*_h )  \RA A \in  \+ D(\sim_p)$ for each $p>0$; equivalently, $\Gamma(A) =0$.  \ec \end{thm}

\part{Reverse mathematics}

The following two sections have been written by Ludovic Patey in March 2015.

\section{Patey: Pseudo Ramsey's theorem for pairs}

Pseudo Ramsey's theorem for pairs has been introduced by Murakami, Yamazaki
and Yokoyama in~\cite{Murakami2014Ramseyan}. They proved that it is between the chain antichain principle
and the ascending descending sequence principle, and asked whether
it was equivalent to one of them. We answer positively.

\begin{definition}[Ascending descending sequence]
$\mathsf{ADS}$ is the statement ``Every linear order has an infinite ascending or descending sequence''.
\end{definition}

\begin{definition}[Pseudo Ramsey's theorem]
A coloring~$f : [\mathbb{N}]^2 \to 2$ is \emph{semi-transitive}
if whenever~$f(x, y) = 1$ and~$f(y, z) = 1$, then~$f(x, z) = 1$ for $x < y < z$.
A set~$H = \{x_0 < x_1 < \dots \}$ is \emph{pseudo-homogeneous} for a coloring~$f : [\mathbb{N}]^n \to k$
if~$f(x_i, \dots, x_{i+n-1}) = f(x_j, \dots, x_{j+n-1})$ for every~$i, j \in \mathbb{N}$.
$\mathsf{psRT}^n_k$ is the statement ``Every coloring~$f : [\mathbb{N}]^n \to k$ has an infinite pseudo-homogeneous set''.
\end{definition}

\begin{theorem}
$\mathsf{RCA}_0 \vdash \mathsf{psRT}^2_2 \leftrightarrow \mathsf{ADS}$
\end{theorem}
\begin{proof}
The direction~$\mathsf{psRT}^2_2 \rightarrow \mathsf{ADS}$ is Theorem~24 in~\cite{Murakami2014Ramseyan}.
We prove that~$\mathsf{ADS} \rightarrow \mathsf{psRT}^2_2$. 
Let~$f : [\mathbb{N}]^2 \to 2$ be a coloring. The reduction is in two steps.
We first define a $\Delta^{0,f}_1$ semi-transitive coloring
$g : [\mathbb{N}]^2 \to 2$ such that every infinite set pseudo-homogeneous for~$g$
computes an infinite set pseudo-homogeneous for~$f$.
Then, we define a $\Delta^{0,g}_1$ linear order~$h : [\mathbb{N}]^2 \to 2$
such that every infinite set pseudo-homogeneous for~$h$
computes an infinite set pseudo-homogeneous for~$g$.
We conclude by applying~$\mathsf{ADS}$ over~$h$.

\emph{Step 1}:
Define the coloring~$g : [\mathbb{N}]^2 \to 2$ for every~$x < y$ by
$g(x, y) = 1$ if there exists a sequence~$x = x_0 < \dots < x_l = y$
such that $f(x_i, x_{i+1}) = 1$ for every~$i < l$, and~$g(x, y) = 0$ otherwise.
The function~$g$ is a semi-transitive coloring. Indeed,
suppose that~$g(x, y) = 1$ and~$g(y, z) = 1$, witnessed respectively
by the sequences~$x = x_0 < \dots < x_m = y$ and~$y = y_0 < \dots < y_n = z$.
The sequence~$x = x_0 < \dots < x_m = y_0 < \dots < y_n = z$ witnesses~$g(x, z) = 1$.
We claim that every infinite set~$H = \{x_0 < x_1 < \dots \}$ pseudo-homogeneous for~$g$ computes
an infinite set pseudo-homogeneous for~$f$. If $H$ is pseudo-homogeneous with color~0,
then $f(x_i, x_{i+1}) = 0$ for each~$i$, otherwise the sequence~$x_i < x_{i+1}$ would witness
$g(x_i, x_{i+1}) = 1$. Thus $H$ is pseudo-homogeneous for~$f$ with color~0. 
If $H$ is pseudo-homogeneous with color~1,
then define the set~$H_1 \supseteq H$ to be the set of integers in the sequences
witnessing~$g(x_i, x_{i+1}) = 1$ for each~$i$. The set~$H_1$ is~$\Delta^{0,f \oplus H}_1$
and pseudo-homogeneous for~$f$ with color~1.

\emph{Step 2}:
Define the coloring~$h : [\mathbb{N}]^2 \to 2$ for every~$x < y$ by
$h(x, y) = 0$ if there exists a sequence~$x = x_0 < \dots < x_l = y$
such that $g(x_i, x_{i+1}) = 0$ for every~$i < l$, and~$h(x, y) = 1$ otherwise.
For the same reasons as for~$g$, $h(x, z) = 0$ whenever~$h(x, y) = 0$ and~$h(y, z) = 0$ for~$x < y < z$.
We need to prove that if~$h(x, z) = 0$ then either $h(x, y) = 0$ or~$h(y, z) = 0$ for~$x < y < z$.
Let~$x = x_0 < \dots < x_l = z$ be a sequence witnessing $h(x,z) = 0$.
If~$y = x_i$ for some~$i < l$ then the sequence~$x = x_0 < \dots < x_i = y$ witnesses~$h(x,y) = 0$.
If~$y \neq x_i$ for every~$i < l$, then there exists some~$i < l$ such that~$x_i < y < x_{i+1}$. 
By semi-transitivity of~$g$, either $g(x_i, y) = 0$ or~$g(y, x_{i+1}) = 0$. In this case either~$x = x_0 < \dots < x_i < y$ witnesses~$h(x,y) = 0$ or~$y < x_{i+1} < \dots < x_l = z$ witnesses~$h(y,z) = 0$.
Therefore~$h$ is a linear order.
For the same reasons as for~$g$, every infinite set pseudo-homogeneous for~$h$
computes an infinite set pseudo-homogeneous for~$g$.
This last step finishes the proof.
\end{proof}

\section{Patey: Increasing polarized Ramsey's theorem for pairs}

The Ramsey-type weak K\"onig's lemma has been introduced by Flood in~\cite{Flood2012Reverse}
under the name $\mathsf{RKL}$, and later renamed $\mathsf{RWKL}$ by Bienvenu, Patey and Shafer.
Independently, the increasing polarized Ramsey's theorem has been introduced by Dzhafarov
and Hirst~\cite{Dzhafarov2009polarized} to find new principles between stable Ramsey's theorem for pairs
and Ramsey's theorem for pairs. We prove that the two principles are equivalent over $\mathsf{RCA}_0$.

\begin{definition}[Ramsey-type weak K\"onig's lemma]
Given an infinite set of strings $S \subseteq 2^{<\mathbb{N}}$,
let~$T_S$ denote the downward closure of $S$, that is,
$T_S = \{ \tau \in 2^{<\mathbb{N}} : (\exists \sigma \in S)[\tau \preceq \sigma] \}$.
A set $H \subseteq \mathbb{N}$ is \emph{homogeneous} for a $\sigma \in 2^{<\mathbb{N}}$ if $(\exists c < 2)(\forall i \in H)(i < |\sigma| \rightarrow \sigma(i)=c)$, and a set $H \subseteq \mathbb{N}$ is \emph{homogeneous} for an infinite tree $T \subseteq 2^{<\mathbb{N}}$ if the tree 
$\{\sigma \in T : \text{$H$ is homogeneous for $\sigma$}\}$ is infinite.
$2\mbox{-}\mathsf{RWKL}$ is the statement~``For every set of strings~$S$,
there is an infinite set which is homogeneous for~$T_S$''.
\end{definition}

\begin{definition}[Increasing polarized Ramsey's theorem]
A set~\emph{increasing p-homogeneous} for $f : [\mathbb{N}]^n \to k$
is a sequence~$\langle H_1, \dots, H_n \rangle$ of infinite sets such that
for some color~$c < k$, $f(x_1, \dots, x_n) = c$ for every increasing
tuple $\langle x_1, \dots, x_n \rangle \in H_1 \times \dots \times H_n$.
$\mathsf{IPT}^n_k$ is the statement~``Every coloring~$f : [\mathbb{N}]^n \to k$
has an infinite increasing p-homogeneous set''.
\end{definition}

\begin{theorem}
$\mathsf{RCA}_0 \vdash \mathsf{IPT}^2_2 \leftrightarrow 2\mbox{-}\mathsf{RWKL}$
\end{theorem}
\begin{proof}
$\mathsf{IPT}^2_2 \rightarrow 2\mbox{-}\mathsf{RWKL}$:
Let~$S = \{\sigma_0, \sigma_1, \dots \}$ be an infinite set of strings
such that~$|\sigma_i| = i$ for each~$i$.
Define the coloring~$f : [\mathbb{N}]^2 \to 2$ for each~$x < y$ by $f(x, y) = \sigma_y(x)$.
By~$\mathsf{IPT}^2_2$, let~$\langle H_1, H_2 \rangle$ be an infinite set increasing p-homogeneous for~$f$ with some color~$c$.
We claim that~$H_1$ is homogeneous for~$T_S$ with color~$c$. We will
prove that the set $I = \{\sigma \in T_S : \text{$H_1$ is homogeneous for $\sigma$}\}$ is infinite.
For each~$y \in \mathbb{N}$, let~$\tau_y$ be the string of length~$y$ defined by
$\tau_y(x) = f(x,y)$ for each~$x < y$. By definition of~$f$, $\tau_y \in S$ for each~$y \in \mathbb{N}$.
By definition of~$\langle H_1, H_2 \rangle$, $\tau_y(x) = c$ for each~$x \in H_1$ and~$y \in H_2$. 
Therefore, $H_1$ is homogeneous for~$\tau_y$ with color~$c$ for each~$y \in H_2$.
As $\{\tau_y : y \in H_2 \} \subseteq I$, the set~$I$ is infinite and therefore~$H_1$
is homogeneous for~$T_S$ with color~$c$.

$2\mbox{-}\mathsf{RWKL} \rightarrow \mathsf{IPT}^2_2$:
Let~$f : [\mathbb{N}]^2 \to 2$ be a coloring.
For each~$y$, let~$\sigma_y$ be the string of length~$y$
such that~$\sigma_y(x) = f(x, y)$ for each~$x < y$,
and let~$S = \{ \sigma_i : i \in \mathbb{N} \}$.
By~$2\mbox{-}\mathsf{RWKL}$, let~$H$ be an infinite set homogeneous for~$T_S$ with some color~$c$.
Define~$\langle H_1, H_2 \rangle$ by stages as follows.
At stage~0, $H_{1,0} = H_{2,0} = \emptyset$.
Suppose that at stage~$s$, $|H_{1,s}| = |H_{2,s}| = s$, $H_{1,s} \subseteq H$
and $\langle H_{1,s}, H_{2,s} \rangle$ is a finite set increasing p-homogeneous for~$f$ with color~$c$.
Take some~$x \in H$ such that~$x > max(H_{1,s}, H_{2,s})$ and set~$H_{1,s+1} = H_{1,s} \cup \{x\}$.
By definition of~$H$, there exists a string $\tau \prec \sigma_y$ for some~$y > x$,
such that $|\tau| > x$ and~$H$ is homogeneous for $\tau$ with color~$c$. Set~$H_{2,s+1} = H_{2,s} \cup \{y\}$.
We now check that the finite set $\langle H_{1,s+1}, H_{2,s+1} \rangle$ is aincreasing p-homogeneous for~$f$ with color~$c$.
By induction hypothesis, we need only to check that $f(z, y) = c$ for every~$z \in H_{1,s+1}$. 
By definition of homogeneity
and as~$H_{1,s+1} \subset H$, $\sigma_y(z) = c$ for every~$y \in H_{1,s+1}$. By definition of~$\sigma_y$,
$f(z, y) = c$ for every~$z \in H_{1,s+1}$. 
This finishes the proof.
\end{proof}

 \part{Higher Randomness}
\section{Yu and Zhu: On $NCR_L$}

This is a joint work of Liang Yu (Heidelberg) and  Yizheng Zhu (M\"unster).

In Logic Blog 2014, Prop.\ 9.4, it was proved that $NCR_L$ is a $\Pi^1_3$-countable set.

We assume $PD$ throughout this section.

\begin{proposition}
If $A$ is a $\Pi^1_2$-countable set, then  $A\subseteq NCR_L$.
\end{proposition}
\begin{proof}
Suppose that $\varphi$ is a $\Pi^1_2$-formula and $x$ is a so that $\varphi(x)$ and the set $\{y\mid \varphi(y)\}$ is countable. Suppose that there is a continuous measure $\rho$ so that $x$ is $L$-random respect to $\rho$. Note that $L[\rho,x]\models \varphi(x)$ by the Shoenfield absoluteness. Then $p\Vdash \varphi(\dot{x})$ for some condition $p$. Then  for any  $L$-random real $y\in p$,  $L[\rho,y]\models \varphi(y)$. By Shoenfield absoluteness again, $\varphi(y)$ is true which contradicts to the assumption.
\end{proof}
Note that there is a contructible real which does not belong to any $\Pi^1_2$-countable set. 

\begin{definition}
  $Q_3=\{x\mid \exists \alpha<\omega_1\forall z(|z|=\alpha\rightarrow x\leq_{\Delta^1_3}z)\}.$
\end{definition}

$Q$-theory was introduced and studied by Harrington, Kechris, Martin, Solovay and Woodin.

\begin{proposition}
For any real $x$, there is a real $y\geq_T x$ so that there is a continuous measure $\rho\leq_T y$ so that $y$ is $L$-random respect to $\rho$.
\end{proposition}
\begin{proof}
For any $x$, let $r$ be $L[x]$-random. Then $y=x\oplus r$ is $L$-random respect to  $\rho$ for the follow continuous measure $\rho\leq_T x\oplus r$.

$\rho(\sigma^{\smallfrown}i)= \left\{
\begin{array}{r@{\quad\quad}l}
\frac{\rho(\sigma)}{2}, & |\sigma|\mbox{ is odd},  \\
\rho(\sigma),& |\sigma|\mbox{ is even} \wedge i=x(\frac{|\sigma|}{2}),\\
0, & \mbox{Otherwise.}\\
\end{array}\right.$

\end{proof}

Let $$B=\{y\mid  \exists \rho\leq_T y(\rho \mbox{ is a continuous measure and $y$ is $L$-random respect to }\rho)\}.$$ Then $B$  has cofinally many $L$-degrees. 
\begin{proposition}
$2^{\omega}\setminus B$ is uncountable.
\end{proposition}

Actually $B$ is $\Pi^1_2$. 

Let $D=\{y_0\mid \forall y(y\geq_T y_0\rightarrow y\in B)\}$. Then $D$ is a nonempty $\Pi^1_2$-set and so contains a real $z_0<_T y_{0,3}$ but $y_{0,3}\not\leq_{Q_3} z_0$which is a base for $D$, where $y_{0,3}$ is the $Q_3$-complete real.

  \begin{lemma}For any $y\not\leq_{Q_3} z_0$, $y\oplus z_0$ is $L$-random respect to some continuous measure $\rho\leq_T y\oplus z_0$. Further more,  $y$ must be $L$-random respect to some continuous measure.\end{lemma}
 \begin{proof}
 Suppose that $y\not\leq_{Q_3} z_0$, then by Posner-Robinson Theorem relative to $z_0$ (due to Woodin), there is a real $z_1\geq_T z_0$ so that $y\oplus z_1\geq_T y_{0,3}^{z_1}$. Then by the discussion above relative to $z_1$, $y$ is $L$-random respect to some measure $\rho\leq_T y\oplus z_1$.
 \end{proof}
 
So $NCR_L\subseteq \{r\mid r\leq_L z_0\}$.

Since $NCR_L\subseteq \mathcal{C}_3$, we have that $NCR_L\subseteq Q_3$.

So we have the following result.

\begin{theorem}
$NCR_L$ is a proper subset of $Q_3$.
\end{theorem}

Let $P_2=\{x\mid \forall y(\kappa^x\leq \kappa^y\implies x\leq_L y)\}$, where $\kappa^x=((\aleph_1)^+)^{L[x]}$. The following lemma is obvious.
\begin{lemma}\label{lemma: p2 subset ncrl}
$P_2\subseteq NCR_L$.
\end{lemma}

Now let $M_1$ be the minimal model with a Woodin cardinal. By Steel's result, $2^{\omega}\cap M_1=Q_3$.  
\begin{theorem}\label{theorem: ncrl is confinal}
$NCR_L$ is cofinal (in the $L$-degree sense) in $Q_3$.
\end{theorem}

An immediate conclusion of Theorem \ref{theorem: ncrl is confinal} is
\begin{corollary}
$NCR_L$ is not $\Sigma^1_3$.
\end{corollary}

Note that, by 6E14 in Msochvakis book, $NCR_L \cap \Delta^1_3$ is cofinal in $\Delta^1_3$.

%%%%%%%%%%%%%%%%%%%%%%%%%%%%%%%%
%%%%%%%%%%%%%%%%%%%%%%%%%%%%%%%%
%\part{Higher randomness}
 %%%%%%%%%%%%%%%%%%%%%%%%%%%%%%%%%

 %%%%%%%%%%%%%%%%%%%%%%%%%%%%%%%%%
 %%%%%%%%%%%%%%%%%%%%%%%%%%%%%%%%%
%  \part{Algebra}
%  Tent_group_extensions
% 

\part{Group  theory and its connections to logic}

\section{Tent: Low-tech notes on group extensions}

\newcommand{\inv}{^{-1}}
 
\newcommand{\F}{\mathbb{F}}
\newcommand{\C}{\mathcal{C}}
\newcommand{\G}{\mathcal{G}}
\newcommand{\M}{\mathcal{M}}
\newcommand{\U}{\mathbb{U}}

\renewcommand{\phi}{\varphi}

(By Katrin Tent) We  explain in a ``low-tech" way how to describe and understand
group extensions.  

%This technique,   was later used by Nies and Tent in the final section of~\cite{Nies.Tent:arxiv} to give an a  first-order description of length $O(\log^3|G|)$ for any  finite group $G$  via choosing a decomposition series.  They met at the Research Center Whiritoa and needed this to fix a gap in an earlier version of their paper.  They  describe the members of the series (simple groups) in $O(\log)$ length, and then use the technique to describe how the extensions are constructed.   Their language is different from the one given here to facilitate expressibility in first-order logic.

%:
\subsection{Background}

It is well-known that group extensions of a group $N$ by a group $G$ can be classified  via the second cohomology groups of
certain associated modules. See e.g.\  \cite[Ch.\ 11]{Robinson:82}. Since this theory is quite involved, 
we  give here  an easy description of the class of possible group
extensions $E$ of $N$ by $G$, i.e. groups $E$
containing $N$ (or an isomorphic copy of $N$) as a normal subgroup  such that $E/N\cong G$.
One  writes this as \[1\longrightarrow N\longrightarrow E\longrightarrow G\longrightarrow 1.\]

Suppose that a group $G$ has the presentation
\[G=\langle s_1,\ldots   s_k \mid r_1,\ldots ,r_m\rangle\cong F_k/R\]
where $F_k$ is the free group of rank $k$ on generators
$s_1,\ldots s_k$ and $R$ is the normal subgroup of $F_k$
generated (as a normal subgroup) by $r_1,\ldots, r_m$.

Now suppose we have an extension $E$ such that
\[1\longrightarrow  N\longrightarrow E\longrightarrow G=\langle s_1,\ldots s_k\rangle \longrightarrow 1.\]

Let $\hat s_1,\ldots \hat s_k\in E$ be \emph{lifts} of $Rs_1,\ldots, Rs_k\in G$, i.e. the canonical projection of $\hat s_i$ is $Rs_i$ for $ i=1,\ldots, k$.
Clearly the $\hat s_i$ act on $N$ by conjugation and hence any \emph{word} $w=w(s_1,\ldots, s_k)$ in the
free group $F_k$ with generators $s_1,\ldots,s_k$ acts as an automorphism of $N$ via the natural conjugation 
action of $w(\hat s_1,\ldots,\hat s_k)\in E$.
Hence any group extension $E$
of $N$ by a $k$-generated group $G=\langle s_1,\ldots, s_k\rangle$ comes with an action of $F_k=F(s_1,\ldots, s_k)$ on $N$.

\subsection{Describing $E$}
Towards  describing  $E$ we will have to express this $F_k$ action on $N$. But this is not sufficient. We are missing the natural maps from $R$ to $N$ that transfer from the ``view"  $F_k/R$ of $G$ to the view $E/N$. Define $\phi_E: R\longrightarrow N$ by $ w(s_1,\ldots,s_k)\mapsto w(\hat s_1,\ldots,\hat s_k)$ for $w(s_1,\ldots,s_k) \in R$. By $ \Hom_{F_k}(R,N)$ we denote the set of homomorphisms from $R$ to $N$ that preserve the $F_k$ action. 
The following   is easy to verify.
\begin{lemma}\label{l:phi_E}
Using the previous notation we have $\phi_E\in\Hom_{F_k}(R,N)$.
\end{lemma}

The next lemma states that the group $E$ is determined -- up to isomorphism over $N$ -- by the
action of $F_k$ on $N$ and the homomorphism $\phi_E$:

\begin{lemma}\label{l:phi_E}
Using the previous notation suppose that  $E^1, E^2$ are groups with a common normal subgroup $N$
such that $E^i/N\cong G, i=1,2$. Let $\hat s^j_i\in E^j,j=1,2, i=1,\ldots, k$ be lifts of $Rs_1,\ldots, Rs_k$, respectively.

Suppose that the induced $F_k$-actions agree, i.e. for all $a\in N$ we have

\[a^{\hat s^1_i}=a^{\hat s^2_i}.\]

Then  $\phi_{E^1}=\phi_{E^2}$ $\LR$ $E^1$ and $E^2$ are isomorphic over $N$ via an isomorphism $g$ such that  $g(\hat s_i^1)=\hat s_i^2, i=1,\ldots k$.
\end{lemma}
\begin{proof}
First suppose that $\phi_{E^1}=\phi_{E^2}$. Define
\[g: E_1\longrightarrow E_2, w(\hat s^1_1,\ldots,\hat s^1_k)n\mapsto w(\hat s^2_1,\ldots,\hat s^2_k)n.\]
Note that \[w(\hat s^1_1,\ldots, \hat s^1_k)n=w'(\hat s^1_1,\ldots,\hat s^1_k)n'\] if and only if
\[w(\hat s^1_1,\ldots,\hat s^1_k)(w'(\hat s^1_1,\ldots, \hat s^1_k))^{-1}\in N\] if and only if
\[w(s_1,\ldots, s_k)(w'(s_1,\ldots,s_k))^{-1}\in R.\] Since $\phi_{E^1}=\phi_{E^2}$, we
see that indeed $g$ is well-defined  and injective.

Note that $f$ is an $F_2$-homomorphism because the $F_2$-actions on $N$ agree. Since $E^j$ is generated by $N$ and $\hat s^j_1,\ldots,\hat s^j_k, j=1,2$, this now implies that $g$ is surjective
and hence an isomorphism.

For the other direction, suppose that $g:E^1\longrightarrow E^2$ is an isomorphism over $N$ with $g(\hat s_i^1)=\hat s_i^2, i=1,\ldots k$
and $g\upharpoonright N=\id$. For any $w(s_1,\ldots, s_k)\in R$ we thus have $\phi_{E^1}(w(s_1,\ldots,s_k))=w(\hat s^1_1,\ldots,\hat s^1_k)=g(w(\hat s^1_1,\ldots,\hat s^1_k))=
w(\hat s^2_1,\ldots,\hat s^2_k)=\phi_{E^2}(w(s_1,\ldots, s_k))$, proving the claim.
\end{proof}

\subsection{Action of $\Hom_F(R,Z(N))$}
In the following to simplify notation we let $k=2$. We consider the role of the $F$-homomorphisms from $R$ to the centre of $N$. 
Recall that a group action is called \emph{regular}
if it is transitive and point stabilizers are trivial. 

\begin{lemma}
Let $C$ be the center of $N$.
The group $\Hom_F(R,C)$ acts regularly on the set \[X=\{\phi_E\colon E\mbox{ is extension of } N \mbox{ by } G \mbox{
with prescribed }  F_2\mbox{-action on } N\}\]
via $\phi^\alpha(w(s,t))=\phi(w(s,t))\alpha(w(s,t))$ for $\alpha\in Hom_F(R,C)$ and $\phi\in X$
\end{lemma}
\begin{proof} It is clear that point stabilizers are trivial. 
To see that the action is transitive,  notice that for extensions $E_1, E_2$ of $N$ by $G$ with the prescribed  $F_2$-action on $N$, and lifts $\hat s_i,\hat t, i=1,2$ as before
we have  for all $n\in N$

\[n^{\phi_{E_1}(w(s,t))}=n^{w(\hat s_1,\hat t_1)}=n^{w(s,t)}=n^{w(\hat s_2,\hat t_2)}=n^{\phi_{E_2}(w(s,t))}\]

and hence $\phi_{E_1}(w(s,t))(\phi_{E_2}(w(s,t)))^{-1}\in C$ and so $\phi_{E_1}$ and $\phi_{E_2}$ differ
by an element in $\Hom_F(R,C)$.

We next verify  that $\phi_E^\alpha=\phi_{E'}$ for some extension $E'$ with the same prescribed $F_2$-action on $N$. 
Define $E'$ by choosing a transversal $T$ for $F_2/R$ so that any element $w(s,t)\in F_2$ can be written uniquely
as $w(s,t)=v(s,t)r(s,t)$ where $v(s,t)\in T,r(s,t)\in R$.

We now define the elements of $E'$ as $nw(s,t)=nv(s,t)\phi_E(r(s,t))\alpha(r(s,t))$
with the induced multiplication. Then $E'$ is an extension with the prescribed $F(s,t)$ action  and $\phi_{E'}=\phi_E^\alpha$.
\end{proof}

\section{Doucha and Nies: groups with bi-invariant metric}
\label{s:DouchaNies}
Michal Doucha and Nies worked in Auckland and at the Research Centre Coromandel during Michal's visit to New Zealand December  2014-January 2015. One topic of their discussions was   groups that are equipped with a  bi-invariant metric.  Michal has already submitted or published several papers on this.

\begin{definition} A metric $d$ on a group $G$ is called \emph{bi-invariant}  if the left and the right translations  are isometries, that is,  $d(gx, gy) = d(xg, yg)= d(x,y)$ for each $x,y, g \in G$. \end{definition}

Every group $G$ has a trivial such metric, namely $d(x,y)= 0$ if $x=y$, and $1$ otherwise. 
For a a topological  group $G$,  a natural  question is whether $G$ admits a compatible bi-invariant metric, that is, one which induces the given topology on $G$. Recall that  a topological group   $G$  is metrizable iff $1_G$ has a countable base of neighbourhoods.  A well known useful fact is the following.  Completely metrizable means that the  group   admits some complete compatible metric;  this includes  Polish group in particular. 

\begin{prop} Any compatible bi-invariant metric  $d$ on a completely metrizable group  $G$   is   complete.  \end{prop} 
To see this, let      $\bar{G}$ be the metric completion of $G$ with respect to $d$. Using bi-invariance, one can easily check that for every $x,y,u,v\in G$ we have $d(x,y)=d(x^{-1},y^{-1})$ and $d(xy,uv)\leq d(x,u)+d(y,v)$, thus the inverse operation is isometric and the multiplication is Lipchitz. It follows that the group operations extend to the completion $\bar{G}$ (note that this is not true in general for left-invariant metrics, where the inverse does not always extend to the completion). Next, it is a well-known fact from general topology that a completely metrizable subset of a metrizable space is $G_\delta$. Thus $G$ is a dense $G_\delta$ subset of $\bar{G}$. However, if $G\neq \bar{G}$, then there are left-cosets of $G$ in $\bar{G}$ which are disjoint and still dense $G_\delta$. That contradicts the Baire category theorem.

The following is another well known fact.
\begin{prop} A  topological group   $G$ has a compatible bi-invariant metric iff $1_G$ has a countable base of neighbourhoods  so that each member is closed under conjugation. \end{prop}

 For instance, Abelian Polish groups and compact Polish  groups admit a compatible bi-invariant metric. For another set of  examples, let $M$ be a bounded metric space. Let $G$ be the group of isometries of $M$ with the supremum distance $d(\sigma, \tau ) = \sup_{g \in G} d(g\sigma, g \tau)$. Then $d$ is bi-invariant. Note that this group is in general not separable (it is separable if  $M$ is compact).

$SL_2(\RR)$ is an example of a Polish  group that does not admit a bi-invariant metric \cite[Exercise 2.1.9]{Gao:09}. Given that abelian Polish groups admit a compatible bi-invariant metric, it is natural to ask what happens for groups that in some sense   close to abelian, such as a group that is nilpotent of class $2$.  For instance, the   Heisenberg group $UT^3_3(\RR)$, which   consists of the upper triangular $3 \times 3$ matrices with $1$'s on the main diagonal, is nilpotent of class 2.   We equip this group with the  usual Euclidean topology  of  $\RR^3$.  
\begin{prop}  $UT^3_3(\RR)$ does not admit a compatible bi-invariant metric. \end{prop}
\begin{proof} Suppose that $d$  is a compatible  metric on $UT^3_3(\RR)$. Given matrices

\bc  $A'=\begin{pmatrix}
1 & a' & c'\\
0 & 1 & b'\\
0 & 0 & 1\\
\end{pmatrix}$, $A=\begin{pmatrix}
1 & a & c\\
0 & 1 & b\\
0 & 0 & 1\\
\end{pmatrix}$,   \ec
we have 

$$A^{-1}A'A=\begin{pmatrix}
1 & a' & c'-ab'+a'b\\
0 & 1 & b'\\
0 & 0 & 1\\
\end{pmatrix}.$$ 

Since $d$ is compatible there exists $r>0$ such that  $B^d_1(I)$, the open ball of radius $1$ centred at $I$ with respect to $d$, is contained in the open set of matrices $\begin{pmatrix}
1 & x & z\\
0 & 1 & y\\
0 & 0 & 1\\
\end{pmatrix}$ with $|x|, |y|, |z| < r$. Again since $d$ is compatible, we can choose $A'$ with $a' \neq 0$ and    $d(A',I)<1$. Choosing  $b$ sufficiently large we have  $|c'-ab'+a'b|>r$. It follows that $d(A^{-1}A'A,I)\geq 1$. So $d(A',I)\neq d(A^{-1}A'A,I)$, whence    $d$ is not  bi-invariant.
\end{proof}

There however are non-abelian groups with finite-dimensional Euclidean topology equipped with a compatible bi-invariant metric. Consider for example the group $U(n)$ of unitary $n\times n$ matrices. This group is compact and hence has a bi-invariant metric. To be more concrete, it can for instance be   equipped with the Hilbert-Schmidt norm $\|\cdot\|_{HS}$, i.e. for $u\in U(n)$ we have $\|u\|_{HS}=\sqrt{\sum_{i,j} |u_{i,j}|^2}$. Then the corresponding distance is bi-invariant. Indeed, notice that the definition of the norm can be written using the trace, i.e. $\|u\|=\sqrt{\mathrm{tr}(u^*u)}$. Then direct calculation shows, using the property that $\mathrm{tr}(t^*ut)=\mathrm{tr}(u)$, that for any $u,v,t,t'\in U(n)$ we have $$\mathrm{tr}((tut'-tvt')^*(tut'-tvt'))=\mathrm{tr}((u-v)^*(u-v)).$$

Can the two   cases when a Polish group has a bi-invariant metric be combined?
\begin{question} Suppose the Polish group  $G$ has a closed abelian normal subgroup $A$ such that the Polish group $G/A$
is compact. Does $G$ admit a compatible bi-invariant metric? \end{question}

For finite groups, there is an interesting question: 

\begin{question} Does  the class $\+ C$ of  finite groups with bi-invariant metric form a Fraisse class? \end{question} 
In effect we are asking whether $\+ C$ has the amalgamation property (AP).
The AP  is known for finite groups by an old result of Hall (or possibly Neumann?). It seems to be unknown  as well for finite groups equipped with a   left (say) invariant metric.

\section{Nies: descriptive set theory for profinite groups}
A compact topological group $G$ is called profinite if  the clopen sets form a basis for the topology.  Equivalently, the open normal subgroups form a  base of neighbourhoods of  the identity. Since open subgroups of a compact group have finite index, this means that  $G$ is the inverse limit of its  system of finite quotients with the natural projection maps. 
%That is, its topology is a Stone space.

In a group that is finitely generated as a profinite  group, all subgroups of finite index are open \cite{Nikolov.Segal:07}. This deep theorem implies that the topological structures is determined by the group theoretic structure. In particular, all abstract homomorphisms between such groups are continuous.

Let $\hat F_k$ be the free profinite group in $k$ generators ($ k \le \omega$). This is the inverse limit of the inverse system   $\seq{F_k/N}$  with the canonical maps, where $N$ ranges over the normal subgroups of finite index. (More generally, every countable residually finite group $G$ is embedded into a profinite group $\hat G$ in this way.)

 A presentation of a profinite group has the form \[ \hat F_k / N\]
where $N$ is a closed normal subgroup of $\hat F_k$.  One can also think of a presentation as an ``expression" 
\begin{equation} \la x_1, x_2, \ldots \mid r_1, r_2, \ldots \ra \label{eq:presentation} \end{equation}
where the list of generators $x_1, \ldots$ has length $k$, and the list of relators $r_i \in \hat F_k$  has length at most $\omega$, and of course  the list of relators is finite for the  case of f.p.\ profinite groups. We will see below that the two views are equivalent in a ``Borel" way. This means that the equivalence is carried out by functions that are Borel between      suitable Polish spaces of presentations on the one hand, and closed normal subgroups of $\hat F_k$ on the other hand; isomorphism is preserved in both directions. Generally, if $X,Y$ are Polish spaces and $E, F$ equivalence relations on $X,Y$ respectively, one writes $X,E \le_B Y,F$ (or simply $E \le_B F$) if there is a Borel function $g \colon X \to Y$ such that $Euv \lra F gu gv$ for each $u,v \in X$.

%However, not every f.g.

\begin{prop} The isomorphism relation $E_{f.g.}$ between  finitely generated profinite groups is Borel equivalent to $\mathtt{id}_\RR$, the identity equivalence relation  on $\RR$. The same holds for the isomorphism relation $E_{f.p.}$ of finitely \emph{presented} profinite groups. \end{prop}

   An equivalence relation that is Borel-below $\mathtt{id}_\RR$ is called \emph{smooth}. We thank Alex Lubotzky for pointing out the crucial fact in \cite[Prop.\ 2.2]{Lubotzky:01} used below to show this smoothness of the isomorphism relation.
   
\begin{proof} To show that $\mathtt{id}_\RR \le_B E_{f.p.}$, we use the argument  of Lubotzky~\cite[Prop 6.1]{Lubotzky:05} that there are continuum many non-isomorphic  profinite groups that are f.p.\ as profinite groups. For a set $P$ of primes  let $$G_P= \prod_{p\in P} \mathtt{SL}_2(\ZZ_p)= \mathtt{SL}_2(\hat \ZZ)/ \prod_{q\not \in P} \mathtt{SL}_2(\ZZ_q).$$
Here $\ZZ_p$ is the profinite ring  of $p$-adic integers, and $\hat \ZZ$ is the completion of $\ZZ$, which is isomorphic to $\prod_{p \, \text{prime}} \ZZ_p$. The second equality shows that $G_P$ is finitely presented as a profinite group. Clearly the map  $P \to G_P$ is Borel, and $P= Q \lra G_P \cong G_Q$. 

We now show that $E_{f.g.} \le_B \mathtt{id}_\RR$. We note that Silver's dichotomy theorem, e.g.\ \cite[5.3.5]{Gao:09},  implies that any equivalence relation strictly Borel below $\mathtt{id}_\RR$ has  countably many classes. So  the plain result of Lubotzky~\cite[Prop 6.1]{Lubotzky:05} now already yields  the Borel equivalence $E_{f.p.} \equiv_B E_{f.g.} \equiv_B \mathtt{id}_\RR$. However, by the proof of the result explained  above,  we in fact don't need Silver's result.

The idea is as follows. At first  let us only consider  presentations in a  fixed number $k$ of generators. Let $ \+ N(\hat F_k)$ be the Polish space of normal closed subgroups of $\hat F_k$ (detail below). The Polish   group $G= \mathtt{Aut}(\hat F_k)$ acts continuously  on $\+ N(\hat F_k)$.   For $S, T \in \+ N(\hat F_k)$, we have \[ \hat F_k/S \cong \hat F_k/ T \lra  \ex \theta \in G \, [ \theta(S) = T] \]
by \cite[Prop.\ 2.2]{Lubotzky:01} (which uses a profinite version of Gasch\"utz's Lemma on lifting generating sets of finite groups).  Note that  $G$ is compact, and in fact, profinite \cite[Ex.\ 6 on page 52]{Wilson:98} (but not f.g.)  So its orbit equivalence relation on  $ \+ N(\hat F_k)$ is closed, and hence smooth; see e.g.\  \cite[5.4.7]{Gao:09}.
% (else add some dummy variables to one)
\end{proof}

We give some detail on the  descriptive set theory. To see that   $ \+ N(\hat F_k)$   is a Polish space, note that the compact subsets of a compact metric  space  $M$ form a complete metric  space $\+ K(M)$  with the Hausdorff distance. This space is   compact as well, and is  the completion of the space of finite subsets of $M$. In the case of a  compact Polish group $G$, which we equip with a bi-invariant metric $d$ as explained in Section~\ref{s:DouchaNies}, the normal subgroups form a closed subset of $\+ K(G)$. Firstly, if $\seq {U_n}\sN n$ is a sequence of subgroups converging to $U$, then $U$ is a subgroup: if $a,b \in U$, for each $\epsilon $, for large $n$  we can choose $a_n, b_n \in U_n$  with $d(a_n, a) < \epsilon$ and $d(b_n, b) < \epsilon$. Then $a_nb_n \in U_n$ and $d(a_nb_n, ab) < 3\epsilon$. Secondly, since the metric is bi-invariant, if  all the $U_n$ are normal in $G$, then so is $U$. 

(One can avoid compactness of $G$, as long as there is a bi-invariant  metric compatible with the topology: the closed normal subgroups of a Polish group $G$ form a Polish space, being a closed subset of the Effros space $\+ F (G)$ of non-empty closed sets in $G$. While this space is usually seen as  a Borel structure, it can be topologized using  the Wijsman topology,  the weakest topology that makes all the maps $C \to d(g,C)$, $C\in \+ F (G)$, continuous.)

The presentations in the second sense above form a Polish space $\hat F_k^{\le \omega}$. Let us check that one can pass between the two views of  presentations in a Borel way. For one direction, given $\hat F_k / N$, by the Selection Theorem of Kuratowski/Ryll-Nardzewski (e.g.\ \cite[14.1.4]{Gao:09}, we can pick a countable dense subset $r_1, r_2, \ldots$ of $N$ in a Borel way. 

For the converse direction,  suppose we are given a presentation in the sense of (\ref{eq:presentation}). We have to find in a Borel way the least closed normal subgroup $N$ of $\hat F_k$ containing all the relators $r_i$.  Since $\+ K (\hat F_k)$ is the completion of the metric space of  finite subsets of $\hat F_k$ under the Hausdorff distance, it suffices to  find for each $n$   a finite subset $V_n$ of $\hat F_k$ such that $d(V_n, N) \le 1/n$: then $\seq {V_n}\sN n$ is a Cauchy sequence converging to $N$. 

%Michal: I just didn't understand how that argumentation leads to the proof that the map sending the sequence of relators to a closed normal subgroup is Borel. It seemed incomplete to me. As I understand it you can produce a countable dense subset of the closed normal subgroup from the sequence of relators in a Borel way. For example, you can find countably many Borel functions from the sequence of relators to the free profinite group such that the union of these images is dense in the normal subgroup they generate. And you also don't need compactness for that. 
%

%:
 Let $X$ be the abstract group generated by the   $x_i$, and let $R$ be the countable subgroup of $\hat F_k$ that is generated by all the conjugates of the relators by elements of $X$. Then $\ol R = N$. By compactness of $\hat F_k$, for each $n$ we can find in a Borel way a finite $V_n \sub R$ such that $d(V_n, R) \le 1/n$, and hence $d(V_n, N) \le 1/n$ as required.

\subsection{Complexity of isomorphism for separable profinite groups}  
We now consider profinite groups that    aren't necessarily finitely generated.
As before, we think of such a group as being given by a   presentation $ \hat F_k / N$, $N$ closed, but  now always  $k= \omega$. Note that now one has to explicitly require that   isomorphisms are continuous (while this was automatic for f.g.\ groups). 

The $N \in \+ N(\hat F_\omega)$ with $(\hat F_\omega) ' \sub N $  (commutator subgroup) form a closed subset of $\+ N(\hat F_\omega)$. To see this, note that it suffices to require that the (dense) countable group $(F_\omega)' $ is contained in $N$. This is the space $N_{ab} (\hat F_\omega)$ of  presentations of separable  abelian  profinite groups.  

 As pointed out by A.\ Melnikov, even isomorphism of these  groups  is quite complex.   Pontryagin  duality  (see e.g.\ \cite{Hofmann.Morris:06}) is a functor on the category of abelian locally compact   groups that associates  to each $G$     the  group $  G^*$ of continuous homomorphisms from $G$  into the circle $\mathbb T$,  with the compact-open topology (which coincides with  the topology inherited from the product topology if $G$ is discrete). For a morphism $\alpha \colon G \to H$ let $\alpha^* \colon H^* \to G^*$ be the morphism defined by $\alpha^*(\psi) = \alpha \circ \psi$. 
 
 The Pontryagin  duality theorem says that      $ G\cong ({G^*})^* $  via the application map, for each locally compact abelian group $G$.  A special case of this states that (discrete) abelian   torsion  groups $A$ correspond to  abelian   profinite groups (see \cite[Thm.\ 2.9.6]{Ribes.Zalesski:00} for a self-contained proof of this special case).  Then, as $A$ ranges through the abelian countable torsion  groups, $A^*$ ranges through the separable abelian  profinite groups, with $A \cong B$ iff $B^* \cong A^*$. So  isomorphism for abelian countable torsion groups is Borel equivalent to continuous isomorphism of  separable  abelian profinite groups. 
 
 The former  isomorphism relation  is closely related to the equivalence relation $\mathtt{id}(2^{< \omega_1})$ discussed in \cite[Section 9.2]{Gao:09}, which is modelled on the  classification of  countable abelian torsion groups via Ulm invariants.  Also see the diagram \cite[p.\ 351]{Gao:09} which shows that $\mathtt{id}(2^{< \omega_1})$ is strictly   between $\mathtt{id}_\RR$ and graph isomorphism.

 Nies has shown that isomorphism for general separable profinite groups is $S_\infty$-complete, which means it is of the same Borel complexity as isomorphism of countable graphs. The hardness part uses a construction of Alan Mekler   \cite{Mekler:81} coding ``nice" countable graphs $A$ into countable nilpotent groups $G(A)$  of class 2 and exponent an odd prime $p$.  The idea is now to replace $G(A)$ by  a certain profinite completion, and show that the graph can still be recovered by an interpretation.

 A computable Polish  group is given by a computable  metric  space together with computable group operations on it. The groups $\ZZ_p$ and $\hat F_k$ we discussed  above are computable.
 
 \begin{question} Determine the complexity of isomorphism for f.g.\ computable profinite groups. \end{question} 

\section[Khelif: a free metabelian groups]{Khelif: a free metabelian group of rank at least 2 is \\ bi-interpretable with the ring of integers}

(Translated from French and slightly expanded by Nies.)

Let $G$ be a free metabelian group of rank at least 2. Anatoly Khelif (ca.\ 2006) has shown that $G$ is bi-interpretable (see e.g.\ \cite[Ch.\ 5]{Hodges:93} or \cite{Nies:DescribeGroups})  with $(\NN, + \times)$. This implies that  $G$ and $\NN$ have the same model theoretic properties. For one thing, $G$ is a prime model of its theory. It also implies  that $G$ is quasi-finitely axiomatizable  in the sense of Nies~\cite{Nies:03}: there is a first-order sentence $\phi$ of which  $G$ is, up to isomorphism, the only finitely generated model. Another example of a quasi-finitely axiomatizable  group is the Heisenberg group $UT^3_3(\ZZ)$ over the integers. However, Khelif has shown that $UT^3_3(\ZZ)$   is not bi-interpretable with $\ZZ$ even with parameters  (see Thm.\ 7.16 in \cite{Nies:DescribeGroups}).  $G$ is also interesting because it is  a quasi-finitely axiomatizable group that is not finitely presented. In \cite{Nies:03} the first example    of such a group was obtained: the restricted wreath product $(\ZZ/ p\ZZ)  \wr \ZZ$.  See the 2007 survey~\cite{Nies:DescribeGroups} for background. 

\begin{thm}[A.\ Khelif] Let $G$ be a free metabelian group of rank $m \ge  2$. The group $G$ is bi-interpretable (in  parameters)  with the ring of integers. \end{thm}
\begin{proof} The commutator subgroup $G'$ is definable in $G$  by the formula $\phi(x) \equiv \fa y \, [x,x^y]=1$ by a result of Mal'cev \cite{Malcev:60}. (It is easy to see that each commutator $[p,q]$ satisfies the formula, because $[p,q]^g$ is also a commutator.)

Now let $H = [G,G']$ be the third term of the descending central series of $G$. The subgroup $H$ is also definable  as the products of at most $m$ commutators of the form $[x,y]$ where $x \in G, y \in G'$. 

Let $S = G/H$. Let $CC(x)$ denote the bi-commutator of $x$, namely centralizer of the centralizer of $x$ (which contains all the powers of $x$). We say that $u \in S$ is \emph{primitive} if it satisfies the formula
\[ \fa v \in CC(u)  \fa w  \ex t  \, [u,t] = [v,w].\]
It is clear that a minimal set of generators of the   group $S$ only contains primitive elements.

\subsection{The copy of $\ZZ$ defined in $S$}
Unless otherwise noted quantifiers range over $S$. We say that a pair of elements $(u_0, u_1 )$ in $S$ is \emph{admissible} if $u_0$ is primitive and $u_1 \in CC(u_0)$. On the set $A$ of  admissible pairs we define an equivalence  relation by 
\[  (u_0, u_1) \sim (v_0, v_1)  \lra   \fa w, t [u_0, w]= [v_0, t] \to [u_1,w = v_1,w]  \]

\begin{claim} $(u_0, u_1) \sim (v_0, v_1)$ iff $\ex k \in \ZZ \, ( u_1 = u_0^k \lland v_1 = v_0^k)$. \end{claim}

So we can view $A/\sim$ as the domain of the copy of $\ZZ$. For the operations, note that  for every pair $(u', v') $ and $u$ there is $v$ such that $(u'v') \sim (u,v)$. Given that, it suffices to note  that   the following first-order definitions are compatible  with $\sim$:

$(u,v_0) \oplus (u,v_1) = ( u, v_0v_1)$

$(u,v_0) \otimes (u,v_1) = ( u, w)$ where $\fa r, s ([u,r] = [v_0,r] \to [u,v_1] = [u,w])$.

\begin{claim} $(A, \oplus, \otimes) /\sim $ is isomorphic to $ (\ZZ, +, \times)$. \end{claim}

\subsection{Defining exponentiation inside $G$} We can     define exponentiation on $S$ internally: given $u \in S$ and  $k \in \ZZ$, represented by  $(v,v^k) \in A$, $u^k$ is the unique element $u'\in S$ such that 
\[ \fa t  \, [u,w ] = [v,t]  \to  [ u', w = v^k, t]. \] 
To do this in $G$ rather than $S = G/H$, let $p \colon G \to S$ be the canonical projection. If $u \not \in H$, $u^k$ is the unique element $u'$  of $CC(u)$ such that $p(u') = p(u) ^k$. If $u \in H$, then $u^k$ is the unique element $u' \in H$ such that for each $v \in G - H$, there is $w \in H$ with $(uv)^k = u' [v,w]$.

\subsection{Defining in $G$ the  module structure of $G'$}

The conjugation action of $G$ on $G'$ introduces automorphisms of the abelian group $G'$. They  commute pairwise because $x^[g,h] =x $ for each $g,h \in G, x \in G'$.  Let $u_1, \ldots, u_m$ be a minimal set of generators for $G$, and let $f_i$ be the automorphism of $G'$ induced by $f_i$.  Then $G'$ is a $\ZZ[X_1, \ldots X_m, X_1^{-1}, \ldots X_m^{-1}$ module in a canonical way. This module is free with generators all the  $[u_i,u_j]$ for $i\neq j$.

We can now define the multiplication of an element of the ring and an element of $G'$. Let $P$ be a polynome in  $\ZZ[X_1, \ldots X_m$ and $u \in G'$.  Then $P\cdot u$ is the element $v \in G'$ such that for each $m$-tuple of integers  $\lambda_1 , \ldots , \lambda_m$, $v - P(\lambda_1 , \ldots , \lambda_m)u$ is in the submodule of $G'$ generated by the images of the $f_i^{-\lambda_i}$.  (...) We can therefore define a bijection between $G'$ and $\NN$. Since every element of $G$ is expressed uniquely   as a product of elements in $CC(u_1), \ldots, CC(u_m)$ and $G'$, one gets a definable bijection (in the sense of $G$ with parameters $u_1, \ldots, u_m$)  between $\NN$ and $G$. 
 \end{proof}

%  \part{Metric spaces and descriptive set theory}
% %%%%%%%%%%%%%%%%%%%%%%%%%%%%%%%%%
%\input{sections2015/Dec2014NiesSolecki} 
% 

 \part{General topics}
\n  {\bf Principles common throughout  mathematics.}
On the occasion of     talks to    mathematics and general logic audiences in M\"unster and Paris, Andr\'e Nies  thought  about a unifying approach to mathematics. The goal was to isolate ideas and principles that occur in lots of areas perceived to be disconnected. 

\section{Objects of  greatest complexity in their class}

The following situation arises in many areas of mathematics  and theoretical computer science. Given a class of objects, together with a method to compare their complexity, is there a most complicated object in the class? Is such an object uniquely determined?

 \bi \item  The halting problem is a most complicated object in the class of   computably enumerable sets under many-onereducibility $\le_m$. It is unique up to computable permutations of $\NN$ by Myhill \cite{Myhill:55}. 

\item The satisfiability problem SAT is a most complicated object in NP (Cook/Levin, 1971/1973) under polynomial time many-one reduci\-bility $\le^p_m$.

\item Chaitin's $\Omega$ is  complete for left-c.e.\ reals under Solovay reducibility $\le_S$.  It can be described up to Solovay equivalence $\equiv_S$ as  the unique left-c.e.\ ML-random  \cite{Kucera.Slaman:01}.

 \item There is a most complicated  $K$-trivial set with respect to  ML-reduci\-bility by \cite{Bienvenu.Greenberg.ea:nd}.

\item There is a complete object in the class of countable abelian groups with embedding of structures.

 \item  Conjugacy of ergodic transformations is $\Sigma^1_1$-complete \cite{Foreman.Rudolph.etal:11}.
 
\item 
Isomorphism of separable $C^*$-algebras is  Borel is complete for orbit equivalence relations (\cite{Elliott.ea:nd} together with \cite{Sabok:13}). The same is true for homeomorphism of  compact metric spaces \cite{Zielinski:14}. \ei

Suppose  we are given a preordering $\le$ to compare the complexity of objects in the class $\+ C$. Why are we interested in complete objects $S$ for a class $\+ C$?  The answer depends on whether   we start with $\+ C$,  or with $S$. 

\bi \item[1.] \emph{Starting with $\+ C$.} The preordering is often very simple, and should definitely be simpler that the objects it is supposed to compare. (In fact it is sometimes not even mentioned explicitly.)  Usually $\+ C$ is also downward closed under $\le$. So the single object $S$, together with  the simple preordering $\le$, describes the whole class $\+ C$.

\item[2.] \emph{Starting with $S$.} A complete object $S$ for $\+ C$ is   often interesting on its  own right. Reflecting Tao and others, one can ask the question:  is $S$ random, or structured? For instance, Chaitin's $\Omega$ and the Rado graph are random. The halting problem and SAT are structured.

 The complexity of $S$ is   completely determined by proving it is complete for the natural class  $\+ C$ of objects it belongs to.  It is less clear how its randomness content is related to $\+ C$. 
\ei

We now  give some detail   for each example. Then we return to the general metamathematical goal,  by discussing analogs of Post's problem which stems from  computability theory.

\subsection{C.e.\ sets and the halting problem} Sets of natural number can be compared via many-one ($m$, for short) reducibility: $B \le_m A$ if $B= \ES$ or $B = f^{-1}(A)$ for some computable function $f$. Let $W_e$ be the $e$-th c.e.\ set. One version of the halting problem is  the effective join of all the $W_e$, which is clearly   $1$-complete. By the Myhill isomorphism theorem~\cite{Myhill:55},  an $m$-complete set is 
unique  up to a computable permutation of $\NN$.  

Being $m$-complete is equivalent to being creative. Interestingly, creativity   is first-order  definable in the lattice of c.e.\ sets by a result of Harrington (see \cite{Soare:87} or \cite[1.7.20]{Nies:book}).

\subsection{Languages in $\NP$ and the satisfiablity problem $\SAT$} Languages in complexity theory can be compared via polynomial time  $m$-reducibility: $B \le^p_m A$ if $B= \ES$ or $B = f^{-1}(A)$ for some polynomial time computable function $f$.  The Cook-Levin theorem from  the early 1970s  says that $\SAT$ is $\NP$-complete. 

The 1976 Berman-Hartmanis Conjecture asks whether   all $\NP$-complete sets are polynomial time isomorphic. Hundreds of $\NP$-complete problems have been studied. They  are all ``paddable".  Any two paddable polytime $m$-equivalent sets   are polytime isomorphic. However,  these problems are all ``natural"- it is not clear if this can be considered as evidence for the conjecture.  Mahaney~\cite{Mahaney:80} proved that sparse sets (i.e.\ with polynomial upper density) cannot be NP-complete unless P=NP. Is this also known  for sets that merely have  sub-exponential  upper density?

\subsection{Left-c.e.\ reals and Chaitin's $\Omega$} A left-c.e.\ real $\beta$ is given by  $\beta = \sup_s \beta_s$, where $\seq{\beta_s}\sN s$ is a nondecreasing, computable  sequence of rationals. Solovay reducibility is defined by  $\beta \le_S \alpha$  if for  given  computable approximations  $\seq{\beta_s}\sN s$ of $\beta$ and  $\seq{\aaa_s}\sN s$ of $\aaa$, there is  a computable increasing  function $g$ such that $\beta - \beta_{g(s)} = O(\alpha- \alpha_s )$.  Equivalently, there is a left-c.e.\ real $\gamma$   such that   $\tp{-d} \beta + \gamma = \alpha$. See \cite[3.2.28]{Nies:book} or  the monumental monograph \cite[Section 9.1]{Downey.Hirschfeldt:book}.  

An example of an $\le_S$ complete c.e.\ real   is  $\sum_e \tp{-e}\beta_e$, where $\seq{\beta_e}$ is an effective listing of the left-c.e.\ reals in $[0,1]$. Unlike the halting problem, we don't get any uniqueness other than being Solovay complete.  However, \Kuc\  and Slaman \cite{Kucera.Slaman:01} gave a description of the class of $S$-complete left-c.e.\ reals as the ones that are \ML\ random.  

\subsection{$K$-trivials}  There is no largest $K$-trivial with respect to  $\leT$, because each $K$-trivial is low:  \cite[Theorem 5.3.22]{Nies:book}   can be used to build a c.e.\ $K$-trivial not below a given low c.e.\ set.   

In a sense, Turing reducibility~$\leT $ is too fine for a meaningful complexity analysis of the $K$-trivials.  
We define ML-reducibility by  $B \le_{ML} A$ if for each ML-random $Y$, $Y \ge_T A$ implies that  $Y \ge_T B$. 

Bienvenu et al.\ \cite{Bienvenu.Greenberg.ea:nd}   have shown that some set $A$, which they call  a ``smart $K$-trivial", is complete for $\le_{ML}$ within the $K$-trivials. They define Oberwolfach randomness and show that the ML-randoms failing this stronger randomness property are precisely the ones computing all $K$-trivials. Then they build a c.e.\ $K$-trivial such  that no  ML-random Turing above it is Oberwolfach random. No direct characterisation of the class of smart $K$-trivials is known at present. 

\subsection{Structures under embedding}
We look at a class of countable structures under embedding $\preceq$.   Complete structures in this setting are often called \emph{universal}.

\bi \item For (symmetric) graphs there is a complete structure, the Rado (or random) graph. 

\item For linear orders, $(\QQ, <)$ is complete. \item For abelian  groups, there is a complete countable group $A$.  This is because the f.g.\ abelian groups have the amalgamation property. So one can build  $A$ as a Fraisse limit.

\item There is no  countable group that is $\preceq$ complete for  countable groups. For, there are continuum many non-isomorphic 2-generated groups (Higman, Neumann, and Neumann).   Only countably many can be isomorphic to a subgroup of  a given single countable group.
\ei

\begin{question}  Does every variety (in the sense of universal algebra) have  the amalgamation property  for finitely generated structures? \end{question} In that case, if there are only countably many f.g.\ structures in the variety, there is a $\preceq$-complete countable structure, namely  the Fra\'iss\'e limit of the f.g. structures. 
%For instance, this works for abelian groups.
\vsp

The substructure relation $B \preceq A$ doesn't always  say that $B $ is less complex than $A$. The larger structure can ``erase" information from $B$, for instance  when  a linear order \ $B$ is embedded into the dense linear order $B \times \QQ$ (with the lex ordering). All the complexity now lies in the embedding. If $B$ is complicated, then  the  range of the embedding into  $B \times \QQ\cong \QQ$ is also complicated. 

If we require  the embeddings to be in some sense effective, this is no longer possible, and a meaningful theory of relative complexity emerges. We will study this in the simple case that the structure is a set with an equivalence relation  (ER). Firstly we consider the case of domain $\NN$, with computable embeddings. Thereafter we proceed to the case that the domain is an uncountable Polish space. 
 
\subsection[Complexity of equivalence relations]{Complexity of equivalence relations on $\NN$, and presentations of groups}

  We  define $m$-reducibility between ER  by  $ F \le_m E $ if there is a computable  function $g\colon \NN \to \NN$ such that $Fuv \lra E g(u) g(v)$.

\subsubsection*{$\PI n$ completeness} Ianovski et al.\  \cite{Ianovski.Miller.ea:14} showed that there is an $m$-complete $\PI 1$ equivalence relation, and no complete $\PI n$ equivalence relation for $n \ge 2$.   For a natural  example given by structures,  isomorphism of certain polynomial time computable trees   is a  complete $\PI 1$ equivalence relation by \cite{Ianovski:12}.  It is unknown whether  the $\PI 1$ equivalence relation of isomorphism of automatic ER is a complete; for background on this question see \cite{Kuske.Liu.ea:13}.

It is easy to build a $\SI n$ complete ER for each $n$:  they can be listed effectively, so it is sufficient to take the disjoint sum. How about  $\SI n$ completeness for naturally occurring equivalence relations? 

\subsubsection*{$\SI n$ completeness}  In computability theory,  Ianovski et al.\ \cite{Ianovski.Miller.ea:14} showed completeness at the relevant level for a number of degree equivalences on the c.e.\ sets. For instance,   $\equiv_T$  among c.e.\ sets  is $\SI 4$ complete.

Consider a finitely axiomatised variety $\+ V$ of groups, such as all groups, or the metabelian groups. It is easy to see that   isomorphism of $\+ V$-finitely presented groups is $\SI 1$. 

C.F.\ Miller \cite[p.\ 80]{Miller:71} has proved that isomorphism of   finitely presented groups is $m$-complete for $\SI 1$ equivalence relations. We don't know of similar results for more restricted varieties, such as the groups that are solvable at a fixed level.

 We can describe a f.g.\ nilpotent group by $F_c(n)/N$ where $F_c(n)$ is the  free nilpotent group of class $c$ and rank $n$, and $N$ is finitely generated as a normal subgroup. 
 
Nilpotency   is a $\SI 1$ property of a finite presentations of a  group. This can be seen as follows: for having nilpotency class $c$, it   is sufficient  that the generators $g_1, \ldots , g_k$ satisfy   the finitely many relations $[x_1,\ldots, x_{c+1}]=1$ (for in that case, all the $[x_1, \ldots, x_c]$ are in the centre, and so,  if $c>1$,  inductively $G/Z(G)$ is nilpotent of class $c-1$).  
 This  is a $\SI 1$ event.

 So isomorphism is $\SI 1$ as well. We can effectively  list all the presentations of nilpotent groups as $P_0, P_1, \ldots$ and see isomorphism as a relation among the $P_k$. 
 
 Of course, we can also take a finite presentation of $F_c(n)$, and add its relators to a finite presentation of a group in $n$ variables. The   versions of the isomorphism problem for class $c$-nilpotent groups we obtain by describing nilpotent groups in two different ways are $m$-equivalent. 
 
 Isomorphism of abelian f.g.\ groups is decidable. 
 \begin{question} For $c>1$ is   isomorphism of f.g.\ class $c$-nilpotent groups decidable? \end{question}

\subsection{Completeness for preorders} $m$-reductions between  c.e.\ preorders and the corresponding   completeness notions  have been studied beginning   with   \cite{Montagna.Sorbi:85}, and later e.g.\ in \cite{Ianovski.Miller.ea:14}. Implication of sentences under PA is $\SI 1$-complete \cite{Montagna.Sorbi:85}, and weak truth table reducibility on c.e.\ sets is $\SI 3$ complete \cite{Ianovski.Miller.ea:14}.   

\begin{question} The substructure relation  $G \preceq H$ among f.p.\ groups is merely  $\SI 2$ by definition. Is it properly $\SI 2$? 

\n The relation that $G$ is a retract of $H$ is $\SI 1$.  Is it $\SI 1$-complete as a preorder?
\end{question}
For examples of  $\Sigma^1_1$-complete ER on $\NN$, see Part IV of  the 2013 Logic Blog~\cite{LogicBlog:13}.

\subsection{Preliminaries: Choquet theory}   We now move on to examples of completeness in descriptive set theory. First some preliminaries. Choquet\footnote{Gustave Choquet was a student of the analyst  Arnaud Denjoy at ENS Paris in the 1930s.} theory starts out with a locally convex topological vector space $V$ over $\RR$. Such a vector space has a basis of the topology consisting of the  translations of convex sets $C$ that are \emph{balanced} (if $x \in C$ then $\lambda x \in C$ for each $|\lambda| \le 1$), and \emph{absorbent} ($V = \bigcup_n nC$).   This generalises the situation of balls in $\RR^n$.  

For instance, a normed space with the weak topology is locally convex. More generally, given any vector space $V$ and a collection $\+ F$ of linear functionals on it, $V  $  can be turned  into a locally convex topological vector space by giving it the weakest topology that makes all the  linear functionals in  $\+ F$ continuous.  

Locally convex topological vector spaces, more general than the  normed spaces, allow us to  study  interesting compact sets, such as the closed  unit ball in $W^*$ with the weak  $*$ topology, for a   Banach space $W$. (Banach himself  in the year 1932 showed that the unit ball is compact in this topology   for separable Banach space $W$ via a diagonalization argument not relying on the axiom of choice.  Alaoglu proved it   in  full generality  in his 1938  thesis at the  Univ.\ of Chicago, using Tychonoff's theorem, which needs the axiom of choice.)

  Consider a   compact convex set  $C \sub V$.   The set of \emph{extreme points} $E$  is the set of points $x$ in $S$ such that  $2x = y_0+y_1$ implies that $x = y_0=y_1$.  Given a vector space $W$, a  function $f \colon C \to W$ is \emph{affine} if $f(\frac 1 2 (x_0 + x_1)) = \frac 1 2 (f(x_0) + f(x_1))$. 
  
%  The whole thing started with a theorem of Minkowski saying that every point in a compact convex subset of $\RR^n$ is a  convex combination of finitely many extreme points.
  
Let   $V $ be  the dual space of a   Banach space $W$. As mentioned, this is locally convex with the weak $*$ topology, and  the closed unit ball of $V$ is convex and compact. For $W=\+ C [0,1]$  this unit ball can be seen as the set of     measures $\mu $ on $[0,1]$ with a mass of at most $1$. 
  
  For a related example, consider a dynamical system $\la X,T \ra$ where $X$ is a topological space and $T \colon X \to X$. The  probability measures  on $X$ for which $T $ is invariant form a compact convex set in the space of (signed) Borel measures on $X$, which is the dual space of $\+ C(X)$ with the weak topology. The extreme points are the ergodic measures.  %  form a compact convex set $C$, ergodic measures are the extreme points.)
\begin{definition}  A \emph{Choquet simplex}  is a compact convex set $S\sub V$ with the averaging condition   that every $x \in S$ is the barycentre of a unique probability measure $\mu$ on the set $E$ of extreme points:  $f(x) = \int f   d\mu$ for each continuous   affine   function $f \colon \, S \to \RR$.
\end{definition}

\subsection{Equivalence relations and Polish group actions}  
\label{s: eqrels}
We consider  equivalence relations on Polish spaces. Let $X,Y$ denote Polish spaces and $E,F$ equivalence relations. We  define Borel reducibility by  $ (Y,F) \le_B (X,E)$ if there is a Borel function $g\colon X \to Y$ such that $Euv \lra F g(u) g(v)$.  

Consider a Polish group action $G \curvearrowright X $. The corresponding orbit equivalence relation (OER) is $E^X_G = \{ \la u,v \ra \colon \ex g \in G \, g \cdot u =v\}$.   

An  equivalence relation is \emph{orbit complete} if it is Borel equivalent to  an orbit equivalence relation, and every orbit equivalence relation is Borel reducible to it. 

Separable structures can be encoded in various ways as points in a Polish space.  Polish spaces themselves are given via completion  by a distance matrix on a chosen dense sequence. The space $\+ M$ of all Polish metric  spaces is then a $G_\delta $ subset of $ \RR^{\NN \times \NN}$, namely the  functions satisfying the axioms for metric spaces. The compact metric spaces form a $\PI 3$ subset of $\+ M$ using that for metric spaces,  compact $\LR$  (complete $\&$  totally bounded).

Let $\+ B(H)$ denote  the algebra of continuous operators on separable Hilbert space with the topology given by the operator norm. A \emph{separable $C^*$-algebra} is a closed self-adjoint subalgebra of $\+ B(H)$.
Elliott et al.\ \cite{Elliott.ea:nd}   proved that 
isomorphism of separable $C^*$-algebras (with a suitable encoding as Polish metric structures)  is  Borel below an  orbit equivalence relation. They left open the question of orbit completeness. 

Farah, Toms and Tornquist \cite[Cor.\ 5.2]{Farah_etal:14}  Borel-reduced  the affine homeomorphism relation on Choquet simplices to  isomorphism of separable $C^*$-algebras (in fact, of a subclass, the unital simple AI-algebras).
Sabok~\cite{Sabok:13}  then obtained this orbit completeness by showing   
  that isometry of Polish metric spaces is Borel reducible to affine homeomorphism of Choquet simplices.

Using Sabok's result, Zielinski  \cite{Zielinski:14} proved  that  the  homeomorphism relation $\cong_h$ of  compact metric spaces is orbit complete. Similar to \cite{Farah_etal:14},  he  Borel-reduced  affine homeomorphism of Choquet simplices to $\cong_h$.  

Let $\+ C(X)$ be the space of complex valued continuous functions on $X$. Since for compact spaces we have $X \cong Y \lra \+ C(X) \cong \+ C(Y)$ as $C^*$-algebras, this shows that even isomorphism of commutative $C^*$-algebras is orbit complete. 

An example of an orbit complete OER is the OER obtained from the Borel action of $Iso(\mathbb U)$  on  $F(\mathbb U)$, the Effros algebra of Urysohn space. As a byproduct of   Zielinsky's  result and its proof, one  obtains another   example of an orbit complete OER, possibly    more natural than the  previously known ones.
\bi \item Let $\+ Q$ be Hilbert cube $[0,1]^\NN$ with the standard metric 
\bc $d(\ol x, \ol y)= \sum_n \tp{-n-1} |x_n- y_m|$.  \ec 
\item Let $G$ be the group of autohomeomorphisms of $\+ Q$, which is a Polish group with the metric $d(f,g) = d_\infty (f,g) + d_\infty(f^{-1}, g^{-1})$, where $d_\infty(f,g) = \sup_x d(fx,gx)$.  \item Let $X = \+K(\+Q)$ be the Polish space of compact (i.e.\ closed) subsets of $\+ Q$ with the Hausdorff distance. \ei 
The natural action $ G \curvearrowright X$ has an orbit complete OER $E$. Intuitively, to any compact metric space $M$  one can in a Borel fashion assign     ``small" compact  set $C_M \sub \+ Q$ that is  homeomorphic to $M$,  e.g.\ by using that $\+ Q \times \+ Q$ is homeomorphic to $\+ Q$. The smallness of these sets  implies that any homeomorphism between two of them extends to an autohomeomorphism of $\+ Q$. Thus $M \cong_h N  $ iff $C_M E C_N$. 
\vsp

\subsection{Ergodic theory} 
\label{ss:ergodic}	

\n  A Borel probability
 space is given by a probability measure on the Borel sets of a standard Polish space. Foreman, Rudolph and Weiss \cite{Foreman.Rudolph.etal:11} showed: 
\begin{theorem} \label{th:FRW} Conjugacy of ergodic transformations on a non-atomic Borel probability
 space  is analytic complete.  \end{theorem}

By a result of von Neumann, all non-atomic Borel probability
 spaces are measure theoretically isomorphic to the unit interval $ [0,1]$ with Lebesgue measure $\mu$. So we can restrict ourselves to conjugacy of ergodic transformations in  $\MPT$, the group of measure preserving transformations of $[0,1]$, with    transformations  $T_0,T_1$  identified if they agree outside a  null set.
 
 \subsubsection*{$\MPT$ as a Polish space} To make sense of \cref{th:FRW} we need a Polish topology on $\MPT$.  Here is some background. 
  
  An element  $T $ of  $\MPT$ gives rise to  the unitary operator $U_T$ on   the separable Hilbert space $\+ H= L^2([0,1], \mu)$ such that  $U_T(f) = f \circ T$. Note that the equivalence classes of bounded measurable functions are dense in $\+ H$. 
  
  A unitary operator $U$ is of the form $U_T$ iff \bi \item both $U$ and $U^{-1}$ preserve $L^\infty(X)$, i.e.\ the boundedness of (equivalence classes of) functions, and \item  $U(fg ) = U(f) U(g)$ for bounded $f,g$. \ei  See \cite[Thm. 2.4]{Walters:00}. 
  
The strong operator topology on $\+ B (\+ H)$ coincides with the weak operator topology on the set of unitary transformations $\+ U (\+ H)$. The space $ \+ U(\+ H)$  with this topology is separable; a compatible complete metric is for instance  \bc $d(S,T) = \sum_n \tp{-n-1} [||S(x_n)- T(x_n)|| +  ||S^{-1}(x_n)- T^{-1}(x_n)||$, \ec where $\seq{x_n}$ is a dense sequence in the unit ball of $\+ H$ \cite[I.9B]{Kechris:95}. The conditions above    make $\MPT$ a $G_\delta$ subset of $\+ U (\+ H)$, so it forms a Polish space. One can also directly induce this  topology on $\MPT$ using the Halmos metric, which is analogous to the metric  above:
  \bc $d(S,T) = \sum_n \tp{-n-1} [\mu (S(E_n) \Delta  T(E_n)) +  \mu (S^{-1}(E_n) \Delta  T^{-1}(E_n))]$, \ec
  \n where $\seq {E_n}\sN n$ is a list of sets generating  the $\sigma $-algebra, such as the rational closed intervals. 
  This directly turns  $\MPT$ into a Polish metric space.
 
 An operator $T \in \MPT$ is called \emph{ergodic}  if each $T$-invariant set is null or conull. Ergodicity  is known to be  a $G_\delta$ property on $\MPT$. To see this, one uses that $T$ is ergodic iff the Lebesgue  measure $\mu$  is  an extreme point of  the convex set of probability measures  on $[0,1]$ for which $T$ is invariant.

% (Why?  For instance, one could  use that $T$ is ergodic iff $\fa f, g$ bounded $C-\lim \la T^n f, g \ra = E(f) E(g)$, where $C-\lim $ denotes the Cesaro limit of a sequence.  It's enough to look at a suitable countable  class of bounded functions. This doesn't quite make it $G_\delta$ though. 
% 
\subsubsection*{On the proof of Theorem~\ref{th:FRW} due to  \cite{Foreman.Rudolph.etal:11}} Given a subtree $ B$ of $ \strcantor$, Foreman, Rudolph  and Weiss build an ergodic operator $T_B$ such that $B$ has an infinite branch iff $T_{B}$ is conjugate to its inverse in $\MPT$. 
 
 They list the strings in $B$ as $\seq {\sss_n}$ so that $\sss_ n \prec \sss_k $ implies that $n < k$. Next, they  define sets $W_n(B)$ of words over $\{0,1\}$. If $\sss_n \prec \sss_k$ then all the words in $W_k(B)$ are concatenations of words in $W_n(B)$. Let $W(B) = \bigcup_n W_n(B)$.

 Let $\mathbb K(B)$ be the set of $f \in \{0, 1\}^\ZZ$ such that each block $f\uhr{[u,v]}$ is in $W(B)$. In symbolic dynamics, such a set  is called a sub-shift, namely it is closed and shift-invariant.  Let $T_B$ be the shift on $\mathbb K(B)$. Since the base space $\mathbb K(B)$ is   compact, it   carries  a shift-invariant (non-atomic?) probability measure $\mu$; by choosing the $W_n(B)$ in the right way, they  show that  it is unique, which makes the system $(\mathbb K(B), \mu, T_B)$ ergodic. They then verify that $[B] \neq \ES $ iff $T_B$ is conjugate to~$T_B^{-1}$. 
 
  It is not clear whether the construction is effective, because  they use some probabilistic argument near the end of the 58-page paper.

 By the von Neumann result mentioned above, and the fact that it is a Borel translation, we can assume that $T$ is in $\MPT$. 
 
 The big open question is:

\begin{question} Is the relation $E$ of conjugacy of ergodic measure-preserving transformations   $\le_B$-complete for orbit equivalence relations? \end{question}
Foreman had announced at some point  that $E$  is $\le_B$-hard for OER given by  $S_\infty$-actions; no paper on this has appeared so far.  

\subsection{Post's problem}

Frequently,  objects turn out to be  the most complicated in their class. This fact is familiar from computability theory (a~remote branch of mathematical logic formerly known as recursion theory). Post's problem asked whether there is a c.e.\ set intermediate between the computable sets and the halting problem in the sense of Turing reducibility. The answer was yes. However, natural c.e.\ sets that aren't outright computable usually end up having the same complexity as the halting problem.

In some   areas mentioned above, this is different.  

Subsection~\ref{s: eqrels} on orbit equivalence relations (OER): $S_\infty$ is the Polish group of permutations of $\NN$. Graph isomorphism is $\le_B$-complete for $S_\infty$-OER. This has been coded into lots of other ER, even isomorphism of countable Boolean algebras by Camerlo and Gao~\cite{Camerlo.Gao:01}. On the other hand, as they pointed out, isomorphism of countable torsion abelian groups is not complete. This uses  Ulm invariants, which are certain countable sequences of countable ordinals. The result  was proved by Friedman and Stanley~\cite{Friedman.Stanley:89}. 

 Subsection~\ref{ss:ergodic} on ergodic theory: Instead of conjugacy of ergodic transformations $S,T$ in $\MPT$, one can also consider the weaker relation of conjugacy of $U_S, U_T$ in the unitary group  $\+ U (\+ H)$ (i.e., one allows conjugating by elements that are not necessarily  of the form $U_R$ for any $R$ in $\MPT$). Via spectral theory, one can show that this relation is Borel.

\section{Describing a structure within a class}
 We want to describe a structure in a class up to isomorphism, using an appropriate formal language. Containment in the class is given as an   external condition.
 
For  finite structures in  a fixed finite signature, there  is always a description in first-order logic of length comparable to the size of the structure. An interesting question is   how short such a description can be. Nies and  Katrin Tent \cite{Nies.Tent:arxiv}   answered this question for finite groups,   compressing the group $G$  via  a first-order description of   length $O(\log^3 |G|)$. The   Higman-Sims formula  states that the number of non-isomorphic groups of order 
$p^n$ is $p^{O(n^{5/2}) + 2n^3/27}$.    By a counting argument, this shows that the bound obtained is close to optimal.

The Kolmogorov complexity of a finite mathematical  object is the length of a shortest description within an appropriate universal system of descriptions, such as a universal Turing machine. 
  %This idea combines the descriptive and the computational approaches to complexity mentioned in the introduction.
   If we encode a finite structure by a string, we can apply this measure of complexity; however, it is not invariant under isomorphism. It would be worthwhile to  study the invariant  Kolmogorov complexity $K_{\mathtt {inv}}(G)$ of a finite group $G$, which is defined as  the least Kolmogorov complexity of any $H \cong G$. It is not hard to see that $K_{\mathtt {inv}}(G)$ is bounded above  by the length of a shortest  first-order description of $G$ (plus a fixed additive constant). By the same counting argument, the Higman Sims fmla   implies that  
 for a  $p$-group  $G$,  $K_{\mathtt {inv}}(G)$ and the length of a shortest first-order description   are in fact quite close:  both are of the order  $\log^3|G|$.  What happens if we restrict to   other classes of finite  groups?

Within the class of  finitely generated groups, an interesting question is whether a group can be described at all by a single first-order sentence.  If so we call the  group  quasi-finitely axiomatizable (QFA), a notion introduced in~\cite{Nies:sepgroups}. For instance, this is the case for the Heisenberg group over $\mathbb Z$, and  for the restricted wreath product of a finite cyclic group with $\mathbb Z$ (the latter example is interesting here because it is not finitely presented).

Within  the class of countable structures over a countable signature $S$, there is always a description in $L_{\omega_1, \omega}(S)$, the extension of first-order language that allows  countable disjunctions over a set of formulas with a shared finite reservoir of free variables (Scott). For   the class of separable complete metric spaces, a similar result holds. The most natural logic  here is an extension of  Lipschitz logic for $S$ that allows  countably infinite  disjunctions.

%\def\cprime{$'$} \def\cprime{$'$}

%
%\bibliographystyle{plain}
%
%\bibliography{../bibs/Nies,../bibs/randomness,../bibs/settheory,../bibs/various,../bibs/recursiontheory,../bibs/analysis,../bibs/Kucera,../bibs/modeltheory,../bibs/reverse_maths,../bibs/groups,../bibs/ergodic_theory}

\begin{thebibliography}{10}

\bibitem{Alberti.Csornyei.ea:10}
G.~Alberti, M.~Csornyei, and D.~Preiss.
\newblock Differentiability of {L}ipschitz functions, structure of null sets,
  and other problems.
\newblock In {\em Proceedings of the International Congress of Mathematicians},
  pages 1279--1394. World Scientific, 2010.

\bibitem{Kjos.Lempp.ea:04}
K.~Ambos-Spies, B.~Kjos-Hanssen, S.~Lempp, and T.~Slaman.
\newblock Comparing {DNR} and {WWKL}.
\newblock {\em J. Symbolic Logic}, 69(4):1089--1104, 2004.

\bibitem{Andrews.etal:16}
U.~Andrews, N.~Cai, D.~Diamondstone, C.~Jockusch, and S.~Lempp.
\newblock Asymptotic density, computable traceability, and 1-randomness, 2013.

\bibitem{Andrews.etal:2013}
Uri Andrews, Mingzhong Cai, David Diamondstone, Carl Jockusch, and Steffen
  Lempp.
\newblock Asymptotic density, computable traceability, and 1-randomness.
\newblock Preprint, 2013.

\bibitem{Barmpalias.Miller.ea:12}
George Barmpalias, Joseph~S. Miller, and Andr{\'e} Nies.
\newblock Randomness notions and partial relativization.
\newblock {\em Israel J. Math.}, 191(2):791--816, 2012.

\bibitem{Bate:15}
David Bate.
\newblock Structure of measures in {L}ipschitz differentiability spaces.
\newblock {\em J. Amer. Math. Soc.}, 28(2):421--482, 2015.

\bibitem{Beros2013DNC}
Achilles Beros.
\newblock A {DNC} that computes no effectively bi-immune set.
\newblock {\em arXiv preprint arXiv:1308.1324}, 2013.

\bibitem{Bienvenu.Greenberg.ea:nd}
L.~Bienvenu, N.~Greenberg, A.~Ku{\v{c}}era, A.~Nies, and D.~Turetsky.
\newblock Coherent randomness tests and computing the {K}-trivial sets.
\newblock To appear in J. European Math. Society, 2015.

\bibitem{Bienvenu.Gacs..ea:11}
Laurent Bienvenu, Peter G{\'a}cs, Mathieu Hoyrup, Cristobal Rojas, and
  Alexander Shen.
\newblock Algorithmic tests and randomness with respect to a class of measures.
\newblock {\em Proceedings of the Steklov Institute of Mathematics},
  274(1):34--89, 2011.
\newblock Published in Russian in {\it Trudy Matematicheskogo Instituta imeni
  V.A. Steklova}, 2011, Vol. 274, pp. 41--102.

\bibitem{Brattka.Miller.ea:16}
V.~Brattka, J.~Miller, and A.~Nies.
\newblock Randomness and differentiability.
\newblock {\em Transactions of the {AMS}}, 368:581--605, 2016.
\newblock \url{http://arxiv.org/abs/1104.4465}.

\bibitem{Brendle.Brooke.ea:14}
J.~Brendle, A.~Brooke-Taylor, Keng~Meng Ng, and A.~Nies.
\newblock An analogy between cardinal characteristics and highness properties
  of oracles.
\newblock In {\em Proceedings of the 13th Asian Logic Conference: Guangzhou,
  China}, pages 1--28. World Scientific, 2013.
\newblock \url{http://arxiv.org/abs/1404.2839}.

\bibitem{Camerlo.Gao:01}
R.~Camerlo and S.~Gao.
\newblock The completeness of the isomorphism relation for countable boolean
  algebras.
\newblock {\em Trans. Amer. Math. Soc}, 353:491--518, 2001.

\bibitem{Downey.Hirschfeldt:book}
R.~Downey and D.~Hirschfeldt.
\newblock {\em Algorithmic randomness and complexity}.
\newblock Springer-Verlag, Berlin, 2010.
\newblock 855 pages.

\bibitem{Dzhafarov2009polarized}
Damir~D Dzhafarov and Jeffry~L Hirst.
\newblock {The polarized {R}amsey's theorem}.
\newblock {\em Archive for Mathematical Logic}, 48(2):141--157, 2009.

\bibitem{LogicBlog:13}
A.~Nies (editor).
\newblock Logic {B}log 2013.
\newblock Available at \url{http://arxiv.org/abs/1403.5719}, 2013.

\bibitem{Elliott.ea:nd}
George~A Elliott, Ilijas Farah, Vern Paulsen, Christian Rosendal, Andrew~S
  Toms, and Asger T{\"o}rnquist.
\newblock The isomorphism relation for separable c*-algebras.
\newblock {\em arXiv preprint arXiv:1301.7108}, 2013.

\bibitem{Farah_etal:14}
I.~Farah, A.~Toms, and A.~T{\"o}rnquist.
\newblock Turbulence, orbit equivalence, and the classification of nuclear
  c*-algebras.
\newblock {\em Journal f{\"u}r die reine und angewandte Mathematik (Crelles
  Journal)}, 2014(688):101--146, 2014.

\bibitem{Flood2012Reverse}
Stephen Flood.
\newblock Reverse mathematics and a {R}amsey-type {K\"o}nig's lemma.
\newblock {\em Journal of Symbolic Logic}, 77(4):1272--1280, 2012.

\bibitem{Foreman.Rudolph.etal:11}
M.~Foreman, D.~Rudolph, and B.~Weiss.
\newblock The conjugacy problem in ergodic theory.
\newblock {\em Annals of mathematics}, 173(3):1529--1586, 2011.

\bibitem{Friedman.Stanley:89}
Harvey Friedman and Lee Stanley.
\newblock A {B}orel reducibility theory for classes of countable structures.
\newblock {\em Journal of Symbolic Logic}, 54:894--914, 1989.

\bibitem{Furstenberg:2014}
H.~Furstenberg.
\newblock {\em Recurrence in ergodic theory and combinatorial number theory}.
\newblock Princeton University Press, 2014.

\bibitem{Gacs.Hoyrup:11}
P.~G{\'a}cs, M.~Hoyrup, and C.~Rojas.
\newblock Randomness on computable probability spaces - a dynamical point of
  view.
\newblock {\em Theory Comput. Syst.}, 48(3):465--485, 2011.

\bibitem{Gacs:05}
Peter G{\'a}cs.
\newblock Uniform test of algorithmic randomness over a general space.
\newblock {\em Theoret. Comput. Sci.}, 341(1-3):91--137, 2005.

\bibitem{Galatolo.Hoyrup.ea:10}
S.~Galatolo, M.~Hoyrup, and C.~Rojas.
\newblock Effective symbolic dynamics, random points, statistical behavior,
  complexity and entropy.
\newblock {\em Information and Computation}, 208(1):23--41, 2010.

\bibitem{Gao:09}
Su~Gao.
\newblock {\em Invariant descriptive set theory}, volume 293 of {\em Pure and
  Applied Mathematics (Boca Raton)}.
\newblock CRC Press, Boca Raton, FL, 2009.

\bibitem{HW12}
Denis~R. Hirschfeldt and Rebecca Weber.
\newblock Finite self-information.
\newblock {\em Computability}, 1(1):85--98, 2012.

\bibitem{Hodges:93}
W.~Hodges.
\newblock {\em Model Theory}.
\newblock Encyclopedia of Mathematics. Cambridge University Press, Cambridge,
  1993.

\bibitem{Hofmann.Morris:06}
K.~Hofmann and S.~Morris.
\newblock {\em The Structure of Compact Groups: A Primer for Students-A
  Handbook for the Expert}, volume~25.
\newblock Walter de Gruyter, 2006.

\bibitem{Ianovski:12}
Egor Ianovski.
\newblock Computable component-wise reducibility.
\newblock {\em arXiv preprint arXiv:1301.7112}, 2013.
\newblock MSc thesis, University of Auckland.

\bibitem{Ianovski.Miller.ea:14}
Egor Ianovski, Russell Miller, Keng~Meng Ng, and Andre Nies.
\newblock Complexity of equivalence relations and preorders from computability
  theory.
\newblock {\em The Journal of Symbolic Logic}, 79(03):859--881, 2014.

\bibitem{Jockusch2013Diagonally}
Carl~G Jockusch and Andrew~EM Lewis.
\newblock Diagonally non-computable functions and bi-immunity.
\newblock {\em Journal of Symbolic Logic}, 78(3):977--988, 2013.

\bibitem{Kechris:95}
A.~S. Kechris.
\newblock {\em Classical descriptive set theory}, volume 156.
\newblock Springer-Verlag New York, 1995.

\bibitem{Kjos.ea:2005}
B.~Kjos-Hanssen, W.~Merkle, and F.~Stephan.
\newblock Kolmogorov complexity and the {R}ecursion {T}heorem.
\newblock In {\em STACS 2006}, volume 3884 of {\em Lecture Notes in Comput.
  Sci.}, pages 149--161. Springer, Berlin, 2006.

\bibitem{Kjos.Miller.ea:11}
B.~Kjos-Hanssen, J.~Miller, and R.~Solomon.
\newblock Lowness notions, measure, and domination.
\newblock {\em J. London Math. Soc. (2)}, 84, 2011.

\bibitem{Kucera:85}
A.~Ku{\v{c}}era.
\newblock Measure, {$\Pi\sp 0\sb 1$}-classes and complete extensions of {${\rm
  PA}$}.
\newblock In {\em Recursion theory week (Oberwolfach, 1984)}, volume 1141 of
  {\em Lecture Notes in Math.}, pages 245--259. Springer, Berlin, 1985.

\bibitem{Kucera.Slaman:01}
A.~Ku{\v{c}}era and T.~Slaman.
\newblock Randomness and recursive enumerability.
\newblock {\em SIAM J. Comput.}, 31(1):199--211, 2001.

\bibitem{Kurtz:81}
S.~Kurtz.
\newblock {\em Randomness and genericity in the degrees of unsolvability}.
\newblock Ph.{D.} {D}issertation, University of Illinois, Urbana, 1981.

\bibitem{Kuske.Liu.ea:13}
D.~Kuske, J.~Liu, and M.~Lohrey.
\newblock The isomorphism problem on classes of automatic structures with
  transitive relations.
\newblock {\em Transactions of the American Mathematical Society},
  365(10):5103--5151, 2013.

\bibitem{Lubotzky:01}
A.~Lubotzky.
\newblock Pro-finite presentations.
\newblock {\em Journal of Algebra}, 242(2):672--690, 2001.

\bibitem{Lubotzky:05}
A.~Lubotzky.
\newblock Finite presentations of adelic groups, the congruence kernel and
  cohomology of finite simple groups.
\newblock {\em Pure Appl. Math. Q}, 1:241--256, 2005.

\bibitem{Mahaney:80}
S.~Mahaney.
\newblock Sparse complete sets for np: Solution of a conjecture of berman and
  hartmanis.
\newblock In {\em Foundations of Computer Science, 1980., 21st Annual Symposium
  on}, pages 54--60. IEEE, 1980.

\bibitem{Malcev:60}
AI~Malcev.
\newblock On free solvable groups.
\newblock {\em Soviet Mathematics Dolkady}, 1:65--68, 1960.

\bibitem{MartinLof:68}
P.~Martin-L{\"o}f.
\newblock On the notion of randomness.
\newblock In {\em Intuitionism and Proof Theory (Proc. Conf., Buffalo, N.Y.,
  1968)}, pages 73--78. North-Holland, Amsterdam, 1970.

\bibitem{Mekler:81}
Alan~H Mekler.
\newblock Stability of nilpotent groups of class 2 and prime exponent.
\newblock {\em The Journal of Symbolic Logic}, 46(04):781--788, 1981.

\bibitem{Melnikov.Nies:13}
A.~G. Melnikov and A.~Nies.
\newblock The classification problem for compact computable metric spaces.
\newblock In {\em CiE}, pages 320--328, 2013.

\bibitem{Miller:71}
C.F. Miller.
\newblock {\em On Group-theoretic Decision Problems and Their Classification}.
\newblock Annals of mathematics studies. Princeton University Press, 1971.

\bibitem{Miyabe:11}
K.~Miyabe.
\newblock Truth-table {S}chnorr randomness and truth-table reducible
  randomness.
\newblock {\em Math. Log. Q.}, 57(3):323--338, 2011.

\bibitem{Miyabe:13}
K.~Miyabe.
\newblock {$L^1$-computability, Layerwise computability and Solovay
  reducibility}.
\newblock {\em Computability}, 2:15--29, 2013.

\bibitem{Miyabe.Rute:13}
K.~Miyabe and J.~Rute.
\newblock Van {L}ambalgen's theorem for uniformly relative {S}chnorr and
  computable randomness.
\newblock In {\em Proceedings of the 12th {A}sian {L}ogic {C}onference}, pages
  251--270. World Sci. Publ., Hackensack, NJ, 2013.

\bibitem{Monin.Nies:15}
B.~Monin and A.~Nies.
\newblock A unifying approach to the {G}amma question.
\newblock In {\em Proceedings of Logic in Computer Science (LICS)}. IEEE press,
  2015.

\bibitem{Montagna.Sorbi:85}
F.~Montagna and A.~Sorbi.
\newblock Universal recursion theoretic properties of r.e. preoredered
  structures.
\newblock {\em J. Symbolic Logic}, 50:397--406, 1985
  [\pageref{Montagna.Sorbi:85.1}].

\bibitem{Murakami2014Ramseyan}
Shota Murakami, Takeshi Yamazaki, and Keita Yokoyama.
\newblock On the ramseyan factorization theorem.
\newblock In {\em Language, Life, Limits}, pages 324--332. Springer, 2014.

\bibitem{Myhill:55}
J.~Myhill.
\newblock Creative sets.
\newblock {\em Mathematical Logic Quarterly}, 1(2):97--108, 1955.

\bibitem{Aylwin.Naulin:09}
R.~Naulin and C.~Aylwin.
\newblock On the complexity of the family of compact subsets of {$\mathbb Q$}.
\newblock {\em Notas de matem{\'a}tica}, 5(2):283, 2009.

\bibitem{Nies:sepgroups}
A.\ Nies.
\newblock Separating classes of groups by first--order formulas.
\newblock {\em Intern. J. Algebra Computation}, 13:287--302, 2003.

\bibitem{Nies:DescribeGroups}
A.~Nies.
\newblock Describing groups.
\newblock {\em Bull. Symbolic Logic}, 13(3):305--339, 2007.

\bibitem{Nies:book}
A.~Nies.
\newblock {\em Computability and randomness}, volume~51 of {\em Oxford Logic
  Guides}.
\newblock Oxford University Press, Oxford, 2009.
\newblock 444 pages. Paperback version 2011.

\bibitem{Nies.Stephan.ea:05}
A.~Nies, F.~Stephan, and S.~Terwijn.
\newblock Randomness, relativization and {T}uring degrees.
\newblock {\em J. Symbolic Logic}, 70(2):515--535, 2005.

\bibitem{Nies.Tent:arxiv}
A.~Nies and K.~Tent.
\newblock Describing finite groups by short first-order sentences.
\newblock {\em arXiv preprint arXiv:1409.8390}, 2014, updated 2015.

\bibitem{Nies:03}
Andr{\'e} Nies.
\newblock Separating classes of groups by first-order sentences.
\newblock {\em Internat. J. Algebra Comput.}, 13(3):287--302, 2003.

\bibitem{Nikolov.Segal:07}
N.~Nikolov and D.~Segal.
\newblock On finitely generated profinite groups. {I}. {S}trong completeness
  and uniform bounds.
\newblock {\em Ann. of Math. (2)}, 165(1):171--238, 2007.

\bibitem{Kamo.Osuga:14}
N.\ Osuga and S.~Kamo.
\newblock Many different covering numbers of yorioka’s ideals.
\newblock {\em Archive for Mathematical Logic}, 53(1-2):43--56, 2014.

\bibitem{Pathak.Rojas.ea:12}
N.~Pathak, C.~Rojas, and S.~G. Simpson.
\newblock Schnorr randomness and the {L}ebesgue differentiation theorem.
\newblock {\em Proc. Amer. Math. Soc.}, 142(1):335--349, 2014.

\bibitem{Preiss.Speight:15}
David Preiss and Gareth Speight.
\newblock Differentiability of {L}ipschitz functions in {L}ebesgue null sets.
\newblock {\em Invent. Math.}, 199(2):517--559, 2015.

\bibitem{Ribes.Zalesski:00}
Luis Ribes and Pavel Zalesskii.
\newblock {\em Profinite groups}.
\newblock Springer, 2000.

\bibitem{Robinson:82}
D.~Robinson.
\newblock {\em A course in the theory of groups}.
\newblock Springer--Verlag, 1988.

\bibitem{Rupprecht:10}
N.~Rupprecht.
\newblock Relativized {S}chnorr tests with universal behavior.
\newblock {\em Arch. Math. Logic}, 49(5):555--570, 2010.

\bibitem{Rupprecht:thesis}
Nicholas Rupprecht.
\newblock {\em Effective correspondents to cardinal characteristics in
  {C}icho\'n's diagram}.
\newblock PhD thesis, University of Michigan, 2010.

\bibitem{Rute:13}
J.~Rute.
\newblock {\em Topics in algorithmic randomness and computable analysis}.
\newblock PhD thesis, Carnegie Mellon University, August 2013.
\newblock Available at \url{http://repository.cmu.edu/dissertations/260/}.

\bibitem{Sabok:13}
M.~Sabok.
\newblock Completeness of the isomorphism problem for separable
  ${C}^*$-algebras.
\newblock {\em arXiv preprint arXiv:1306.1049}, 2013.

\bibitem{Schnorr:75}
C.P. Schnorr.
\newblock {\em Zuf\"alligkeit und {W}ahrscheinlichkeit. {E}ine algorithmische
  {B}egr\"undung der {W}ahrscheinlichkeitstheorie}.
\newblock Springer-Verlag, Berlin, 1971.
\newblock Lecture Notes in Mathematics, Vol. 218.

\bibitem{Soare:87}
Robert~I. Soare.
\newblock {\em Recursively Enumerable Sets and Degrees}.
\newblock Perspectives in Mathematical Logic, Omega Series. Springer--Verlag,
  Heidelberg, 1987.

\bibitem{Stephan.Yu:nd}
F.~Stephan and L.~Yu.
\newblock Lowness for weakly 1-generic and {K}urtz-random.
\newblock In {\em Theory and applications of models of computation}, volume
  3959 of {\em Lecture Notes in Comput. Sci.}, pages 756--764. Springer,
  Berlin, 2006.

\bibitem{Terwijn.Zambella:01}
S.~Terwijn and D.~Zambella.
\newblock Algorithmic randomness and lowness.
\newblock {\em J. Symbolic Logic}, 66:1199--1205, 2001.

\bibitem{vLamb:90}
Michiel van Lambalgen.
\newblock The axiomatization of randomness.
\newblock {\em J. Symbolic Logic}, 55(3):1143--1167, 1990.

\bibitem{Walters:00}
Peter Walters.
\newblock {\em An introduction to ergodic theory}, volume~79.
\newblock Springer Science \& Business Media, 2000.

\bibitem{Wilson:98}
J.~S. Wilson.
\newblock {\em Profinite groups}.
\newblock Clarendon Press, 1998.

\bibitem{Zielinski:14}
J.~Zielinski.
\newblock The complexity of the homeomorphism relation between compact metric
  spaces.
\newblock {\em arXiv preprint arXiv:1409.5523}, 2014.

\end{thebibliography}

\end{document}